\documentclass[12pt]{article}
\usepackage[usenames]{color} 
\usepackage{amssymb} 
\usepackage{amsmath} 
\usepackage[utf8]{inputenc} 
\usepackage {amsfonts}
\usepackage {amsthm}
\usepackage{graphicx}
\usepackage{graphics}
\usepackage{hyperref}
\usepackage {framed}
\DeclareUnicodeCharacter{0301}{\hspace{-1ex}\'{ }}

\usepackage{mathtools}
\usepackage[T1]{fontenc}
\usepackage{tipa}
\usepackage{tipx}
\usepackage{wasysym}
\usepackage{comment}

\usepackage {amsxtra}
\usepackage {enumerate}
\usepackage[T1]{fontenc}
\usepackage[table]{xcolor}
\usepackage[top=2.5cm, bottom=2.5cm, left=2.5cm, right=2.5cm]{geometry}
\usepackage{tabstackengine}
\usepackage{tikz-cd}

\usepackage{comment}

\newcommand{\tpp}{\mathrm{tp}}

\newcommand{\minn}{\mathrm{min}}
\newcommand{\ACVF}{\mathrm{ACVF}}

\newcommand{\ACF}{\mathrm{ACF}}

\newcommand{\NIP}{\mathrm{NIP}}
\newcommand{\NTP}{\mathrm{NTP_2}}

\newcommand{\res}{\mathrm{res}}

\newcommand{\Aut}{\mathrm{Aut}}
\newcommand{\acl}{\mathrm{acl}}
\newcommand{\dcl}{\mathrm{dcl}}

\newcommand{\alg}{\mathrm{alg}}

\newcommand{\rv}{\mathrm{RV}}
\newcommand{\val}{\textit{v}}

\newcommand{\st}{\mathrm{St}}

\newcommand{\vf}{\mathrm{VF}}
\newcommand{\vg}{\Gamma}
\newcommand{\rf}{\mathrm{k}}

\newcommand{\A}{\mathcal{A}}

\newcommand{\U}{\mathcal{U}}

\newcommand{\Ualg}{\mathcal{U}_0}

\def\forkindep{\mathrel{\raise0.2ex\hbox{\ooalign{\hidewidth$\vert$\hidewidth\cr\raise-0.9ex\hbox{$\smile$}}}}}

\theoremstyle{plain}
\newtheorem{thm}{Theorem}[section]

\makeatletter
\newtheorem*{rep@thm}{\rep@title}
\newcommand{\newreptheorem}[2]{%
\newenvironment{rep#1}[1]{%
 \def\rep@title{#2 \ref{##1}}%
 \begin{rep@thm}}%
 {\end{rep@thm}}}
\makeatother

\makeatletter
\newtheorem*{rep@cor}{\rep@title}
\newcommand{\newrepcorollary}[2]{%
\newenvironment{rep#1}[1]{%
 \def\rep@title{#2 \ref{##1}}%
 \begin{rep@cor}}%
 {\end{rep@cor}}}
\makeatother

\makeatletter
\newtheorem*{rep@prop}{\rep@title}
\newcommand{\newrepproposition}[2]{%
\newenvironment{rep#1}[1]{%
 \def\rep@title{#2 \ref{##1}}%
 \begin{rep@prop}}%
 {\end{rep@prop}}}
\makeatother

\newreptheorem{thm}{Theorem}
\newrepcorollary{cor}{Corollary}
\newrepproposition{prop}{Proposition}
\newtheorem{lem}[thm]{Lemma}
\newtheorem{prop}[thm]{Proposition}

\newtheorem{cor}[thm]{Corollary}
\newtheorem{claim}{Claim}
\newtheorem{rem}[thm]{Remark}
\newtheorem{fact}[thm]{Fact}
\newtheorem{notation}[thm]{Notation}

\theoremstyle{definition}
\newtheorem{defn}[thm]{Definition}

\newtheorem{exmp}[thm]{Example}

\usepackage{authblk} 
\theoremstyle{remark}

\title{Residually dominated groups in henselian valued fields of equicharacteristic zero}
\title{Residually Dominated Groups in Henselian Valued Fields of Equicharacteristic Zero}
\author{Dicle Mutlu \thanks{University of Leeds}
\ and \ Paul Z. Wang\thanks{Sorbonne Université, LIP6. Previously affiliated with the École Normale Supérieure.}
}

\date{} 

\begin{document}

\maketitle
\begin{abstract} We introduce \emph{residually dominated groups} in pure henselian valued fields of equicharacteristic zero, as an analogue of stably dominated groups introduced by Hrushovski and Rideau-Kikuchi. We show that when $G$ is a residually dominated group, there is a finite-to-one group homomorphism from its connected component into a connected stably dominated group, and we study the functoriality and universality properties of this map. Moreover, we prove that residual domination is witnessed by a group homomorphism into a definable group in the residue field. In our proofs, we use the results of Montenegro, Onshuus, and Simon on groups definable in $\mathrm{NTP}_2$-theories that extend the theory of fields. Along the way, we also provide an algebraic characterization of residually dominated types, generalizing the work by Ealy, Haskell and Simon for stably dominated types in algebraically closed valued fields, and we study their properties.


\end{abstract}

The notion of stable domination was introduced by Haskell, Macpherson, and Hrushovski in \cite{HHM08} for algebraically closed valued fields. It was later extended to \textit{residue field domination} (which we refer to as \emph{residual domination}) in henselian fields of equicharacteristic zero by various authors in different contexts (for example, \cite{EHM19}, \cite{EHS23}, \cite{Vic22} and recently \cite{KRV24}). 


In this paper, we focus on residual domination and definable groups in henselian valued fields of equicharacteristic zero. 
A type $\tpp(a/C)$ is called \textit{residually dominated} if the following condition holds: for any element $b$, if the traces of $Ca$ and $Cb$ in the residue field are algebraically independent, then $\tpp(a/Cb)$ is determined by the type of $a$ over the trace of $Cb$ in the residue field (see Definition \ref{def_res_dom}). 
We work in the language $\mathcal{L}$ introduced in \cite{ACG22}, in which the residue field $\rf$ is expanded by the quotients $\rf^{\times}/(\rf^{\times})^n$ for each $n\in \mathbb{N}$ (see Definition \ref{language}). 
In this setting, the sorts of the residue field and the value group are orthogonal to each other and are stably embedded in a stronger sense (see Remark \ref{stabembedded_with_control_of_parameters}). In the language $\mathcal{L}$, we first give an algebraic characterization of residually dominated types, extending the results of \cite{EHS23} on stably dominated types in algebraically closed valued fields.

We fix a complete theory of a henselian valued field of equicharacteristic zero with a universal model $\U$ in the language $\mathcal{L}$. We denote by $\Ualg$ its algebraic closure, viewed as a model of $\ACVF$, the theory of algebraically closed valued fields of equicharacteristic zero in the language $\mathcal{L}_0$ of geometric sorts.

\begin{repcor}{algebraic_resdom}

 Let $C$ be a field, and $a$ be a tuple in the valued field sort of a henselian valued field of equicharacteristic zero. Let $L$ be the definable closure of $Ca$ in the valued field sort. Assume that $L$ is a regular extension of $C$. Then the following are equivalent:

\begin{itemize}
\item[(i)] $\tpp(a/C)$ is residually dominated.
\item[(ii)] $L$ is an unramified extension of $C$ and $L$ has the separated basis property over $C$. 
\end{itemize}
\end{repcor}
As a consequence, we relate residual domination and stable domination. 

\begin{repthm}{stabdom_eqv_resdom}
    Let $a$ be a tuple in the valued field sort of $\U$ and $C$ be a subfield with $C=\acl(C)$. Then, the following are equivalent:
\begin{enumerate}
    \item The $\mathcal{L}$-type $\tpp(a/C)$ is residually dominated in the structure $\U$,
    \item The $\mathcal{L}_0$-type $\tpp_0(a/C)$ is stably dominated in the structure $\Ualg$.
\end{enumerate}
 
\end{repthm}

We note that one direction of this equivalence follows directly from \cite[Theorem 4.6]{EHS23}, where the valued field is assumed to have a bounded Galois group to ensure quantifier elimination. In our setting, similar to \cite{Vic22}, we rely on the quantifier elimination result from \cite{ACG22}, which does not require any additional assumptions. 


For algebraically closed valued fields, \textit{stably dominated groups} were introduced in \cite{HR19} to study definable groups in the structure. It was shown that stable domination can be witnessed by a group homomorphism into a stable group (see Fact \ref{homomorphismACVF}). In this paper, using Theorem \ref{stabdom_eqv_resdom}, we extend these results to definable groups in henselian valued fields of equicharacteristic zero, using residual domination. To ensure a well-behaved notion of genericity, we assume that the group has a \textit{strongly-$f$-generic} over some model $M$, meaning that it contains a global type such that no translation of it forks over $M$. For $\NTP$- theories, this holds for \emph{definably amenable groups}, which includes abelian groups (see \cite{MOS20}). We then say that a definable group is \textit{residually dominated}, if it has a residually dominated, strongly $f$-generic type over some model $M$.  We show that in $\NTP$-theories, if 
$G$ is a residually dominated group definable in the valued field sort, then its connected component embeds into a connected, stably dominated subgroup of an algebraic group. A key ingredient of our proofs is the results from \cite{MOS20} on definable groups in expansions of fields with an $\NTP$-theory.

\begin{repthm}{HomomorphismResidue} Assume that $T$ is $\NTP$. Let $(G,\cdot)$ be a definable residually dominated group in the valued field sort, with a strongly $f$-generic type over $M$. Then there exist an $\mathcal{L}_0$-definable group $\mathfrak{g}$ in the residue field sort and an $\mathcal{L}$-definable surjective group homomorphism $f: G_M^{00}\to \mathfrak{g}$ such that the strong $f$-generics of $G_M^{00}$ are dominated via $f$. Namely, for each strongly $f$-generic $p$ of $G_M^{00}$ and for tuples $a, b$ in $\U$ with $a \models p|M$, we have $\tpp(b/M f(a))\vdash \tpp(b/Ma)$ whenever $f(a)\forkindep_{\rf(M)}^{alg}\rf(Mb)$.

\end{repthm}
In the proof of Theorem \ref{HomomorphismResidue}, we find a connected stably dominated group $G_1$ into which the residually dominated group \emph{almost} embeds, meaning that there is a finite-to-one map into $G_1$. In the subsequent results, we show the universality and factoriality properties of $G_1$
\begin{repthm}{thm_acvf_enveloppe} Assume that $T$ is an $\NIP$. Let $(G,\cdot)$ be a residually dominated group in the valued field sort. Then, there exists an $\mathcal{L}_0$-type-definable group $G_1$, a definable group homomorphism $\iota: G^{00} \rightarrow G_1$ with finite kernel, with the following properties:
 \begin{enumerate}
        \item The group $G_1$ is, in $ACVF$, connected stably dominated, and the image of any generic of $G^{00}$ is generic, in the sense of $ACVF$, in $G_1$.
        \item For all $\mathcal{L}_0$-type-definable groups $H_1$, for all definable morphisms $f: G^{00} \rightarrow H_1$, there exist $\mathcal{L}_0$-type-definable $H'_1 \leq H_1$, $H_0 \lhd H'_1$ finite, and $\psi: G_1 \rightarrow H'_1/H_0$ such that $f$ factors through $H'_1$ and $\pi \circ f = \psi \circ \iota$, where $\pi : H'_1 \rightarrow H'_1 / H_0$ is the quotient map. In other words, the following diagram commutes:
    \end{enumerate}
\begin{center}
\begin{tikzcd} G^{00}\arrow[d, "f"] \arrow[r, "\iota"]& G_1 \arrow[d, "\psi"] \\
 H'_1\arrow[r, "\pi"]& H'_1 / H_0  \end{tikzcd}
\end{center}
Moreover, such a group $G_1$ is unique up to $ACVF$-definable isogeny.
\end{repthm}

\begin{repprop}{prop_functoriality} Assume $T$ is $\NIP$. Let $(G, \cdot)$, $(H, \cdot)$ be residually dominated definable groups, in the valued field sort. Let $f: G \twoheadrightarrow H$ be a surjective definable group homomorphism. Let $G_1$, $H_1$ be $\mathcal{L}_0$-definable groups, and $\iota_G: G^{00} \rightarrow G_1$, $\iota_H: H^{00} \rightarrow H_1$ be definable group homomorphisms satisfying the conclusions of Theorem \ref{thm_acvf_enveloppe}.

Then, there exists a finite normal subgroup $H_0 \lhd H_1$ and an $\mathcal{L}_0$-definable surjective (in ACVF) group homomorphism $\psi: G_1 \twoheadrightarrow H_1 / H_0$ such that the following diagram commutes:

\begin{center}
\begin{tikzcd} G^{00}\arrow[d, "f"] \arrow[rr, "\iota_G"]&& G_1 \arrow[d, "\psi"] \\
 H^{00} \arrow[r, "\iota_H"]& H_1 \arrow[r] & H_1 / H_0  \end{tikzcd}
\end{center}

\end{repprop}

We note that the equivalence between residual domination and stable domination was recently proved in \cite{KRV24} within a more general framework involving imaginaries and henselian fields with additional structure. Building on these results, we aim to extend our findings to interpretable groups in henselian valued fields with additional structures.

The structure of the paper is as follows. In Section \ref{Preliminaries}, we provide preliminaries on henselian valued fields, stable domination. Moreover we recall some results on definable groups in tame settings. In Section \ref{Residual Domination}, we present our results on residually dominated types, and in Section \ref{Residually dominated groups}, we introduce residually dominated groups and present our main results related to them.

\paragraph{Acknowledgements.}
The first author thanks her supervisor, Deirdre Haskell, for valuable guidance and support; parts of this work originates from the author’s doctoral thesis. Both authors thank Martin Hils and Silvain Rideau-Kikuchi for insightful discussions and for pointing out errors in earlier drafts.
\section{Preliminaries}\label{Preliminaries}
In this section, we review basic facts about henselian valued fields of equicharacteristic zero, as well as stable domination and model theory of groups. We assume that the reader is familiar with the basics of model theory of valued fields and refer to \cite{Mar18} for the subject's background.

\subsection{Henselian Valued Fields of Equicharacteristic Zero}

Given a henselian valued field $(K, v)$ of equicharacteristic zero, let $\Gamma$ denote its value group and $\rf$ its residue field. We expand $\Gamma$ to $\Gamma_{\infty}$ by adjoining the symbol $\infty$, where $\val(0) = \infty$, and for all $\gamma \in \Gamma$, we set $\gamma < \infty$ and $\gamma + \infty = \infty + \gamma = \infty$. Define $\rv^{\times} := K^{\times} / (1 + \mathfrak{m})$, where $\mathfrak{m}$ is the unique maximal ideal of the valuation ring $\mathcal{O}$ of $K$. Then, $\rf^{\times}\subseteq \rv^{\times}$. The valuation map induces a well-defined surjective group homomorphism $\rv^{\times}$, sending $a(1+\mathfrak{m})$ to $\val(a)$. In particular, there is a short exact sequence of abelian groups:
$$
1 \longrightarrow \rf^{\times}
 \xrightarrow{\mathrm{id}} \rv^{\times}
 \xrightarrow{\val_{rv}} \Gamma
 \longrightarrow 0
$$

After adding $0_{\rv}$ to $\rv$ and setting $\val(0_{\rv}) = \infty$, this extends to a short exact sequence of monoids:
$$
1 \longrightarrow \rf
 \xrightarrow{\mathrm{id}} \rv
 \xrightarrow{\val_{rv}} \Gamma_{\infty}
 \longrightarrow 0
$$

Our main results are presented in the language of valued fields in which the residue field is expanded by the sorts $\rf/(\rf^{\times})^{n}$, as introduced in \cite{ACG22}. These are referred to as \emph{power residue sorts} in \cite{Vic22}.
\begin{defn}\label{language} \begin{enumerate}
\item Define $\mathcal{L}_{val}= (\vf, \rf,\vg, \res,\val)$ to be the language of valued fields where $\vf$ and $\rf$ are equipped with the language of rings, $\Gamma$ is equipped with the language of ordered abelian groups together with the symbol $\infty$. Here, $\val: \vf \to \vg$ and $\res: \vf\to \rf $ are the valuation and the residue maps, respectively; we interpret $\val(0):=\infty$ and $\res(a):=0$ whenever $a\not\in\mathcal{O}$.

\item Define $\mathcal{L}$ to be the expansion of $\mathcal{L}_{val}$ by the \textit{power residue sorts} $\mathcal{A}_n= \rf^{\times}/(\rf^{\times})^n \cup\{ \textbf{0}_n \}$. For each $n\in \mathbb{N}$, we add the maps $\pi_n: \rf \to \mathcal{A}_n $ where $\pi_n$ is the projection map on $\rf^{\times}$ and $\pi_n(0):=\textbf{0}_n$. We also add the maps  $\res^n: \vf \to \mathcal{A}_n$ defined as follows: if $\val(a)\in n\Gamma$, choose $b$ with $\val(a)=n\val(b)$ and set $\res^n(a) = \pi_n(\res(\frac{a}{b^n}))$; otherwise set $\res^n(a):=0_n$.

\item Define $\mathcal{L}_{\rv} = (\rv, \rf,\vg, \A, \val_{rv}, \iota, (\rho_n)_{n \in \mathbb{Z}}, (\pi_n)_{n \in \mathbb{N}}$ as the language of short exact sequences of monoids. Here, the sort $\Gamma$ is in the language of ordered groups expanded by $\infty$. The sorts $\rv$ and $\rf$ have the language of monoids $\{\cdot, 1, 0\}$, with the sort $\rf$ further expanded by the sorts $\A$ as described in part $(2)$.  The map $\iota : \rf \to \rv$ is the natural embedding, and $\val_{rv} : \rv \to \Gamma$ is the induced valuation map on $\rv$. For $a \in \rv$, the function $\rho_n$ is defined as follows: if $a \in \val_{rv}^{-1}(n\Gamma)$, then $\rho_n(a) = \rho_n(ab^{-1})$ for some $b \in \rv$ with $\val_{rv}(b) = \val_{rv}(a)$; otherwise, $\rho_n(a) := 0$.

\end{enumerate}

\end{defn}
Here, $\res^n(a)=\rho_n(rv(a))$, and $\res^n$ (respectively, $\rho_n$) is a group homomorphism on $\val^{-1}(n\Gamma)$ (respectively, on $\val_{rv}^{-1}(n\Gamma)$). 

\begin{notation}Let $\rf_{\A}$ denote the $2$-sorted structure $(\rf,\A)$, equipped with the ring language on $\rf$ and the maps $\pi_n$ for each $n\in \mathbb{Z}_{>0}$.\end{notation}
    


\begin{fact}\label{theorem_5.15_ACG22} Let $T$ be a complete theory of henselian valued fields of equicharacteristic zero. Then every $\mathcal{L}_{val}$ formula $\varphi(x,y,z)$ where $x$, $y$ and $z$ are variables belonging to $\vf$, $\vg$, $\rf$, respectively, is equivalent to an $\mathcal{L}$-formula  $\psi(t_1(x),\cdots, t_n(x), y,z)$ where $t_i$'s are one of the following forms: $\val(p(x))$, $\res(\frac{p(x)}{q(x)})$ or $\res^{n}(p(x))$ where $p$ and $q$ are polynomials with integer coefficients.
\end{fact}
As mentioned just before \cite[Theorem 2.23]{Vic22}, the following is a direct consequence of \cite[Theorem 5.15]{ACG22} 
\begin{fact}\label{QE_val_A}(\cite[Theorem 2.23, Corollary 2.24]{Vic22})
Let $T$ be a theory of henselian valued fields of equicharacteristic zero in the language $\mathcal{L}$. Then $T$ eliminates field quantifiers. Moreover, the sorts $\rf_{\mathcal{A}}$ and $\vg$ are purely stably embedded and orthogonal to each other.
\end{fact}

\begin{rem}\label{stabembedded_with_control_of_parameters}
By Fact \ref{theorem_5.15_ACG22}, one can also see that the sorts $\rf_{\A}$ and $\Gamma$ are stably embedded in a stronger sense: if $D$ is a subset of $\rf_{\A}$ (respectively $\Gamma$) over a parameter set $C$ that lies in the sort $\vf$, then there exists an $\rf_{\A}$-formula (respectively, a $\Gamma$-formula) and some $u\in \dcl(C)$ such that $\psi(x, u)$ defines $D$.
\end{rem}

The next result, due to \cite{ACG22}, is restated in our setting.
\begin{fact}\label{QE_RV}(\cite[Corollary 4.8]{ACG22})  The structure $(\rv, \Gamma, \rf)\cup \{\A, \iota, \val, (\rho_{n})_{n\in\mathbb{N}} \}$ eliminates the $\rv$-quantifiers. Moreover, the sorts $\rf_{\A}$ and $\Gamma$ are purely stably embedded and orthogonal to each other. \end{fact}

The following is a standard characterization of henselianity. 

\begin{fact}\label{henselianity} A valued field $K$ is henselian if and only if every valued field automorphism of $K$ extends to a valued field automorphism of its field-theoretic algebraic closure $K^{alg}$.
\end{fact}

A field (with possibly additional structure) is said to be \textit{algebraically bounded} if its model-theoretic algebraic closure is contained in its field-theoretic algebraic closure.

\begin{fact}\label{algebraically_bounded}(\cite{vdD89})
If $(K,v)$ is a pure henselian field of characteristic zero, then $K$ is algebraically bounded.
\end{fact}
The following is a well-known property of henselian valued fields of equicharacteristic zero.
\begin{fact}\label{maximal_immediate_model} Let $T$ be a theory of henselian valued fields of equicharacteristic zero and $M\models T$. Then $M$ has a unique maximal immediate extension and this extension is elementary over $M$. 
\end{fact}

\subsection{Separated Extensions} Let $(C,v)\subseteq (L,v)$ be valued fields, let $V$ be a finite dimensional $C$-vector subspace of $L$. We say $V$ is \textit{separated over $C$} if there exists a $C$-basis $\vec{b}=(b_1,\dots,b_n)$ of $V$ such that for all $c_1,\dots,c_n \in C$, $$\val(\underset{i}{\sum}c_ib_i)=\minn\lbrace v(c_i)+v(b_i)\rbrace_{i}.$$
In this case, $\vec{b}$ is called a \textit{separated basis}. Such a basis is \textit{good} if, in addition, for all $1\leq i,j\leq n$, either $\val(b_i)=v(b_j)$ or $v(b_i)+\Gamma(C)\not=\val(b_j)+\Gamma(C).$ 

    Finally, we say $L$ has \textit{the separated basis property over $C$} (or equivalently, the field extension \textit {$L/C$
is separated}) if every finite-dimensional $C$-vector subspace of $L$ has a separated basis over $C$. We say $L$ has \textit{the good separated basis property over} $C$, if every finite-dimensional $C$-vector subspace of $L$ has a good separated basis over $C$.

In fact, the separated basis property always implies the good separated basis property.
\begin{rem}\label{separatedness_implies_good}
Let $(C,\val)\subseteq (L,\val)$ be valued fields, and suppose that $L$ has the separated basis property over $C$. Then $L$ also has the good separated basis property over $C$.
\end{rem}
\begin{proof} We will prove by induction on the dimension of a finite dimensional $C$-vector subspace of $L$. Let $U$ be an $n$-dimensional $C$-vector subspace with a good separated basis $b_1, \dots, b_n$ and let $b \in L \setminus U$. By assumption, we may assume that $(b_1, \dots, b_n, b)$ is separated over $C$. 

If $\val(b) \not= \val(u)$ for any $u \in U$, then in particular $\val(b) \ne \val(c) + \val(b_i)$ for any $c \in C$ and any $i \leq n$; thus $(b_1, \dots, b_n, b)$ remains a good basis. 

If $\val(b) = \val(u)$ for some $u \in U$, then there exists $i_0 \leq n$ and $c \in C$ such that $\val(b) = \val(c) + \val(b_{i_0})$. After replacing $b$ with $c^{-1}b$, we may assume that $\val(b) = \val(b_{i_0})$, and again $(b_1, \dots, b_n, b)$ is a good basis. \end{proof}

\begin{fact}\label{Delon2}(\cite[Corollary 7]{Del88}) Any algebraic extension of a henselian field is separated.
    
\end{fact}
\begin{fact}\label{Delon1}(\cite[Lemma 5]{Del88}) Let $(C,v)\subseteq (L,v)$ be valued fields. Suppose $U\subseteq V$ are finite-dimensional $C$-vector subspaces of $L$, and $V$ is separated over $C$. Then $U$ is separated over $C$.
\end{fact}
As mentioned in \cite{Del88}, the transitivity of the separated basis property follows from Fact \ref{Delon1}:

\begin{prop}\label{transitivity_of_separated_basis}
    Let $(C,\val) \subseteq (L,\val)\subseteq (M,\val)$ be valued fields such that both $L/C$ and $M/L$ are separated. Then $M/C$ is separated.
\end{prop}
\begin{proof}
  Let $(f_i)_{i \in I}$ be a separated basis of $M$ over $L$ and $(e_j)_{j \in J}$ a separated basis of $L$ over $C$. For any $x \in M$, we can write  $x = \sum l_i f_i = \sum \left(\sum c_{i,j} e_j\right) f_i.$ The products $e_if_j$ are distinct, and the set $\{e_j f_i \mid j \in J, i \in I\}$ is $C$-linearly independent. Moreover, it remains separated over $C$; each $f_i$ is separated over $L$ and for any $c \in C$ and $j\in J$, the product $ce_j$ lies in $L$. 
  Thus, $\{e_j f_i \mid j \in J, i \in I\}$ is a separated basis of $M$ over $C$.

If $U$ is a finite-dimensional $C$-vector subspace of $M$, then $U$ lies in the $C$-span of $e_j f_i$. By Fact \ref{Delon1}, this implies that $U$ is separated over $C$. \end{proof}

We finish the section by recalling a well-known fact.

\begin{fact}\cite[Lemma 12.2]{HHM08}\label{SepBasisMaximal}
    If $C$ is maximally complete, then every valued field extension of C has the separated basis property over $C$.
\end{fact}

\subsection{Stable Domination}\label{stable domination}
In this section, we fix a theory $T$ in some language. We define $\st_C$ to be the multi-sorted structure whose sorts are the $C$-definable sets that are stable and stably embedded with their induced structure. 

With this induced structure, $\st_C$ is stable. For subsets $X,Y$ and $Z$ of $\st_C$, we write $X\forkindep_Y Z$ to denote forking independence of $X$ and $Z$ over $Y$. For a set $A$,  we write $\st_C(A):= \st_C\cap \dcl(A)$.

\begin{defn}
    A type $\tpp(a/C)$ is called stably dominated if for all tuple $b$, if $\st_C(Ca)\forkindep_C\st_C(Cb)$, then $\tpp(b/C\st_C(Ca))\vdash \tpp(b/Ca)$.
 
\end{defn}
If $f$ is a $C$-definable map (or, pro-$C$-definable) into $\st_C$ whose domain contains all realizations of $\tpp(a/C)$, then, we say that $\tpp(a/C)$ is \textit{stably dominated via $f$} if for all tuples $b$, we have $\tpp(b/Cf(a))\vdash \tpp(b/Ca)$ whenever $f(a)\forkindep_C\st_C(Cb)$.

A global type $p$ is \textit{stably dominated over $C$} if $p|C$ is stably dominated.

We now list the properties of stably dominated types in algebraically closed valued fields. 

\begin{fact}\label{stabdom_acl_properties}([Proposition 3.13, Corollary 3.31 (iii), Corollary 6.12]\cite{HHM08}, \cite[Proposition 2.10 (1)]{HR19} ) For all tuples $a$ and parameter sets $C$, the following hold
\begin{enumerate} 
    \item[(i)] Suppose $C=\acl(C)$. If $\tpp(a/C)$ is stably dominated, then it has a unique $C$-definable extension $p$ that satisfies: 
    \begin{itemize}
    \item for all $B\supseteq C$, we have $a\models p|B$ if and only if $\st_C(Ca)\forkindep_C \st_C(B)$.
    \item If $b \models p|C$, then $a \forkindep_C b$ implies $b \forkindep_C a$.
    \end{itemize}
    \item[(ii)] The type $\tpp(a/C)$ is stably dominated if and only if $\tpp(a/\acl(C))$ is.
   \item[(iii)] If $\tpp(a/C)$ is stably dominated and $b\in \acl(Ca)$, then $\tpp(b/C)$ is stably dominated.

\end{enumerate}
\end{fact}


\begin{rem}\label{pushforward_stabdom}
   Fact \ref{stabdom_acl_properties} $(iii)$ in particular implies that if $\tpp(a/C)$ is stably dominated and $f$ is an $C$-definable map, then $\tpp(f(a)/C)$ is stably dominated.
\end{rem}

Stably dominated types are invariant under base change. For the downward direction, we use a recent result of Simon and Vicaria, which avoids relying on invariant global extensions of types over $\acl$-closed sets.
\begin{fact}\label{StabDomExtensions}(\cite[Proposition 4.9]{HHM08}, \cite[Theorem 4.1]{SV24}) Let $p$ be a global $C$-invariant type and let $B\supseteq C$.
\begin{enumerate} 
\item[(i)] If $p|C$ is stably dominated, then so is $p|B$.
\item[(ii)] If $p|B$ is stably dominated, then so is $p|C$.

 
\end{enumerate}
    
\end{fact}

In $\ACVF$, stably dominated types are characterized by being orthogonal to the value group, in the following sense.

\begin{fact}\label{orthogonality_stabdom}
    Let $p$ be $C$-invariant type in $\ACVF$. Then $p$ is stably dominated if and only if for every model $M\supseteq C$ and every realization $a\models p|M$, we have $\Gamma(Ma)=\Gamma(M)$.
\end{fact}

Finally, we restate the algebraic characterization of stable domination from \cite{EHS23}:

\begin{fact}\label{algebraic_stabdom}(\cite[Theorem 3.6]{EHS23}, ) Let $\mathfrak{U}$ be a universal model of the theory of algebraically closed valued fields in the language of geometric sorts. Let $C$ be a subfield and $a$ be a tuple of valued field elements. Let $L$ be the definable closure of $Ca$ in the valued field sort such that $L$ is a regular extension of $C$. Then the following are equivalent:
\begin{enumerate}
    \item $\tpp(a/C)$ is stably dominated,
    \item $L$ is an unramified extension of $C$ and has the separated basis property over $C$.
    \end{enumerate}
    \end{fact}
\subsection{Stabilizers and Group Chunks}
Throughout this section, we fix an arbitrary theory $T$ with a universal model $\U$.

Let $(G, \cdot)$ be a type-definable group in $\U$. For any model $M$, the group $G(M)$ acts on the set $S_G(M)$ of types over $M$ concentrating on $G$ via $g\cdot p = \tpp(g\cdot a/M)$ for $a\models p$. This extends to an action of $G(\U)$ on $S_G(\U)$. Moreover, $G(\U)$ acts on definable subsets of $G$ as follows: if $D$ is a definable subset of $G$ defined by $\phi(x)$ and $g \in G(\U)$, then $g \cdot D=\{x\text{ : }\models \phi(g^{-1}\cdot x)\}$ is the left translate of $D$. The right action of $G(\U)$ on global types in $G$ and on definable subsets of $G$ are defined analogously.
\begin{defn}
    Let $G$ be a definable group and let $p$ be a global type concentrating on $G$. 
    \begin{enumerate}
    \item We say $p$ is an \textit{$f$-generic} of $G$ if for every formula $\varphi(x)\in p$, there exists a small model $M$ such that, for all $g\in G$, the translate $g\cdot \varphi$ does not fork over $M$.
    
    \item   We say that $p$ is a \textit{strongly $f$-generic} of $G$ if there exists a small model $M$ such that every translate $g\cdot p$ does not fork over $M$. 

    \item We say that $p$ is a \emph{definable generic} of $G$ over $M$ if $p$ is $M$-definable and, for all $g \in G(\U)$, the translate $g \cdot p$ is also $M$-definable.
    \end{enumerate}

    \end{defn}



If $\mu$ is an ideal of definable subsets of $G$ and $A$ is a parameter set, we say that $\mu$ is $A$-invariant if it is invariant under all automorphisms fixing $A$ point-wise. Similarly, $\mu$ is $G$-invariant over $A$ if it is invariant under left translations by elements of $G(A)$.

\begin{defn} Let $\mu$ be an $M$-invariant ideal. We say that $\mu$ \emph{has the $S1$-property} if, for any formula $\varphi(x,y)$ and any $M$-indiscernible sequence $(a_i)_{i < \omega}$, if $\varphi(x,a_i) \land \varphi(x,a_j) \in \mu$ for some (or equivalently, any) $i, j < \omega$, then there exists some (or equivalently, any) $i < \omega$ such that $\varphi(x,a_i) \in \mu$.

We say that $\mu$ \emph{has the $S1$-property on an $M$-definable set $A$} if $A \not\in \mu$ and the condition above is satisfied for formulas defining subsets of $A$.

\end{defn}
When $G$ admits a global strongly $f$-generic type over $M$, the set of formulas that do not extend to a strongly $f$-generic type forms an ideal that is both $M$-invariant and $G$-invariant on $M$.
\begin{fact}\label{genericity_ideal} (\cite[Lemma 3.14]{MOS20})
Assume that $G$ admits a global strongly $f$-generic type over $M$. Then, the ideal of formulas that do not extend to a strongly $f$-generic type over $M$ has the S1-property.
\end{fact}

\begin{defn} Let $M$ be a small model of a theory $T$, and let $(G, \cdot)$ be an $M$-type-definable group. Let $\mu$ be the ideal of formulas that do not extend to a global strongly $f$-generic type over $M$. 


\begin{enumerate}
    \item For $p, q\in S_G(M)$, we define $Stab_G(p):=\{g\in G(\U)\textit{ $\mid$ }g\cdot p = p.   \}$, and $Stab_G(p,q):=\{g\in G(\U)\textit{ $\mid$ }g\cdot p = q.   \}$. 
    
    \item  A type $r$ in $G$ is called \textit{wide} if $r$ is not contained in $\mu$. If $p\in S_G(M)$ is an $f$-generic type, we define $$St_{\mu}(p):=\{g\in G(\U) \textit{ $\mid$ } g\cdot p \cap p \text{ is wide.}\}.$$

If $g\in St_{\mu}(p)$, then there exists a realization $a$ of $p$ such that $g\cdot a\models p$, and $\tpp(g\cdot a/Mg)$ extends to a global strongly $f$-generic type over $M$.



    
    We write $Stab_{\mu}(p)$ for the group generated by $St_{\mu}(p)$.

\end{enumerate}
    \end{defn}

In \cite{BYC14}, it was shown that in an $\NTP$ theory, the ideal of formulas that do not fork over $M$ (or, more generally, over any extension base) has the $S1$-property. The converse is also true. The following fact can be found in \cite{Hru12}.
\begin{fact}\label{S_1_wide_implies_nonforking} Assume that $\mu$ is an ideal that is $M$-invariant and has the $S1$ property on an $M$-definable set $A$. Let $p$ be a type concentrating on $A$. If $p$ is $\mu$-wide, then $p$ does not fork over $M$.
\end{fact}
We now summarize some key properties about generics and stabilizers of a group from \cite{CS18} and \cite{MOS20} below. For any definable group $G$, $G_M^{00}$ denotes the smallest $M$-type definable subgroup of bounded index. When $T$ is $\NIP$, one has $G_M^{00}=G
_{\emptyset}^{00}$ for every $M$, and this group is simply denoted by $G^{00}$ (see \cite[Theorem 8.4]{NIPbook}).

\begin{fact}\label{genericity} Let $G$ be a type definable group with a  strongly $f$-generic over $M$, and let $\mu$ be the ideal of definable subsets of $G$ that do not extend to a global strongly $f$-generic type.
\begin{itemize}
    \item[(i)](\cite[Proposition 3.8]{CS18}) If $T$ is $\NIP$, then a global type $p$ is an $f$-generic of $G$ if and only if it has a bounded orbit under the translation by elements of $G$ if and only if $G^{00}=Stab_G(p)=(p\cdot p^{-1})^2$. 
    \item[(ii)](\cite[Theorem 3.18]{MOS20}) If $T$ is $\NTP$ and $p\in S_G(M)$ is an $f$-generic, then $G_M^{00}=Stab_{\mu}(p)=(p\cdot p^{-1})^2$. Moreover, $G_M^{00}\setminus St_{\mu}(p)$ is contained in a union of $M$-definable non-wide sets.
\end{itemize}
    
\end{fact}

If $p \in S(M)$ is $A$-invariant for some $A \subseteq M$, then for any $B \supseteq M$, there is a canonical $A$-invariant extension $p|_{B}$, defined as follows: for a formula $\phi(x,y)$ and $b \in  B^{|y|}$, $\phi(x,b)\in p|_B$ if and only if there is a tuple $m\in M^{|y|}$ with $\phi(x,m)\in p$ and $b\models \tpp(m/A)$. 

If $p(x)$ and $q(y)$ are $A$-invariant types over $M$, then their tensor product $p\otimes q$ is the $A$-invariant type defined by: $(a,b)\models p\otimes q$ if and only if $b\models q$, and $a\models p|_{Mb}$. 

\begin{lem}\label{lemma_indep_products_of_generics}
    Let $G$ be a type-definable group, and $p \in S_G(M)$ be an invariant type, where $M$ is a sufficiently saturated model over which $G$ is defined. Assume that the stabilizer $Stab_G(p)$ of $p$ is $G^{00}$. 

    Let $p '$ denote the invariant type $\tpp(\alpha ' \cdot \alpha^{-1} / M)$, where $\alpha \models p|_M$ and $\alpha ' \models p|_{M \alpha}$. Then, the stabilizer $Stab_G(p')$ of $p '$ is $G^{00}$ as well.
\end{lem}

\begin{proof}
First note that $p '$ is the image of $p \otimes p$ by an $M$-definable map, so it is invariant as well. We wish to show that the stabilizer of $p '$ is $G^{00}$. It suffices to check that it contains $G^{00}$. So, let $N_1 \succeq M$, let $g \in G^{00}(N_1)$, let $\alpha$ realize $p|_{N_1}$ and let $\alpha '$ realize $p|_{N_1\alpha}$, so that $\alpha ' \cdot \alpha^{-1}$ realizes $p '|_{N_1}$. Then, by genericity of $p$, the element $g\cdot \alpha '$ realizes $p|_{N_1\alpha}$. So, by definition of $p '$, the product $g \cdot (\alpha ' \cdot \alpha^{-1})$ realizes $p '|_{N_1}$, as required.
\end{proof}

When dealing with definable groups, and generics, one key question is constructing definable groups and group homomorphisms from generic data. The following result is well-known and uses the existence of definable generics; some form of it was already used in the proof of Proposition 6.2 in \cite{HruPil-GpPFF}. Also, note that this has been improved substantially: Theorem 2.15 in \cite{MOS20} is a technical result which can be used to build definable group homomorphisms, with much weaker hypotheses.

\begin{prop}\label{prop_stabilizer_definable_types}
    Let $T$ be a complete theory, let $G$, $H$ be type-definable groups, and $M$ a sufficiently saturated model of $T$ over which $G$ and $H$ are defined. Let $c_1= (g_1, h_1), c_2=(g_2, h_2), c_3=(g_3, h_3)$ be elements of $G \times H$, with $c_3 = c_2 \cdot c_1$. Assume that the types $\tpp(c_2/ M c_1)$ and $\tpp(c_3 / M c_1)$ are $M$-definable, that $\tpp(g_1 / M)$ is a definable generic of $G$, and that $h_i$ is in $\acl(M g_i)$ for $i=1,2,3$.

    Then, there exists $M$-type-definable subgroups $H_1 \leq H$, and $H_0 \lhd H_1$, with $H_0$ finite, and $H_1$ the normalizer of $H_0$, and an $M$-definable group homomorphism $\psi: G^{00} \rightarrow H_1 / H_0$ such that, for all $\gamma$, $\gamma^{'}$ realizing $\tpp(c_1 / M)$, the element $\pi_H(\gamma^{'} \cdot \gamma^{-1})$ is in $H_1$, and the morphism $\psi$ maps $\pi_G(\gamma^{'} \cdot \gamma^{-1})$ to the coset of $\pi_H(\gamma^{'} \cdot \gamma^{-1})$ in $H_1 / H_0$.
\end{prop}

\begin{proof}\setcounter{claim}{0}
    For $i=1,2,3$, let $p_i = \tpp(g_i / M)$, and $q_i = \tpp(g_i, h_i / M)$. From the hypotheses, we know that $c_1 \in Stab_r(q_2, q_3)$. Thus, for all $\gamma$, $\gamma^{'}$ realizing $\tpp(c_1 / M)$, the product $\gamma^{'} \cdot \gamma^{-1}$ belongs to the $M$-type-definable subgroup $S=Stab_r(q_2) $. We wish to define the morphism $\psi$ using the subgroup $S$. 

    \begin{claim}
        The projection $\pi_G(S)$ of $S$ to $G$ contains $G^{00}$. 
    \end{claim}
    \begin{proof}
        Since $\tpp(g_1 / M)$ is definable, and $h_1 \in \acl(M g_1)$, the type $q_1=\tpp(g_1, h_1 / M)$ is also definable. Let $q^{'}_1$ denote the $M$-definable type which is the image of $q_1 \otimes q_1$ under the multiplication map, as in Lemma \ref{lemma_indep_products_of_generics}. 
        Then, by Lemma \ref{lemma_indep_products_of_generics}, the type ${\pi_G}_*(q^{'}_1)$ is a definable generic of $G$, which concentrates on $G^{00}$. 
        Also, we know that $q_1'$ concentrates on $S=Stab_r(q_2)$, which consists of the elements $g\in G$ with $q_2\cdot g = q_2$. Thus, the type-definable group $\pi_G(S) \leq G$ contains the subgroup generated by ${\pi_G}_*(q_1')$, which is, by the properties above, equal to $G^{00}$. 
    \end{proof}

Now, let $H_0 = \lbrace h \in H \, | \, (1, h) \in S \rbrace$, and let $H_1 \leq H$ be the normalizer of $H_0$.

\begin{claim}
    The group $\lbrace 1 \rbrace \times H_0$ is finite and normal in $S$.
\end{claim}
\begin{proof}
    The fact that $\lbrace 1 \rbrace \times H_0$ is normal in $S$ follows from the definition. Let us prove finiteness. 
    Let $h \in H_0$. To simplify notations, assume that $c_2= (g_2, h_2)$ realizes $q_2|_{Mh}$. Then, the element $(g_2, h_2 \cdot h)$ realizes $q_2|_{Mh}$, in particular, we have $h_2 \cdot h \in \acl(M g_2)$. Hence, we have $h \in \acl(M g_2 h_2)$. Then, the independence $g_2 h_2 \forkindep_M h$ implies that $h \in \acl(M) = M$. Since $h \in H_0$ was arbitrary, this implies the finiteness of $H_0$, as required.
\end{proof}

Thus, since $S$ normalizes $\lbrace 1 \rbrace \times H_0$, and $H_1$ is the normalizer of $H_0$, we have $S \subseteq G \times H_1$, and so, for all $\gamma$, $\gamma^{'}$ realizing $\tpp(c_1 / M)$, the element $\pi_H(\gamma^{'} \cdot \gamma^{-1})$ is in $\pi_H(S) \leq H_1$. Let $T$ be the image of the group $S\cap (G^{00} \times H_1)$ in the quotient $G^{00} \times H_1 / H_0$. Then, the type-definable subgroup $T \leq G^{00} \times H_1 / H_0$ defines the graph of an $M$-definable group homomorphism $\psi: G^{00} \rightarrow H_1 / H_0$. The property of the morphism $\psi$ follows from the fact that for all $\gamma$, $\gamma^{'}$ realizing $\tpp(c_1 / M)$, the product $\gamma^{'} \cdot \gamma^{-1}$ belongs to $S$.    
\end{proof}
The following fact is from \cite{MOS20}, which provides a group homomorphism in theories extending fields within a generalized tame setting. 
\begin{fact}\label{group_chunk_ntp2}(\cite[Theorem 2.19]{MOS20}) Let $T$ be an $\NTP$-theory, extending the theory of fields which is algebraically bounded. Assume that any model of $T$ is definably closed within its field-theoretic algebraic closure. Let $G$ be a group definable in an $\omega$-saturated $M\models T$, and assume that $G$ has a strongly $f$-generic type over $M$. Then, there exist an $M$-definable algebraic group $H$ and an $M$-definable group homomorphism $\iota: G_M^{00}\to H$ with a finite kernel.
\end{fact}

\section{Residual domination}\label{Residual Domination}

From now on, we fix a complete theory $T$ of henselian valued fields of equicharacteristic zero in the language $\mathcal{L}$ with power residue sorts given in Definition \ref{language}. 

For the rest of the paper, we use the following notation.
\begin{notation}\label{notation_main_2}
We will use $\mathcal{L}_0$ to denote the language of geometric sorts and $T_0$ to denote $\ACVF_{(0,0)}$, the theory of algebraically closed valued fields of equicharacteristic zero in $\mathcal{L}_0$. To indicate that we are working within $T_0$, we add a subscript $0$ to the corresponding notation. More precisely,

\begin{enumerate}
    \item $\U$ denotes a universal model of $T$, and $\Ualg$ denotes the universal model of $T_0$ obtained by taking the algebraic closure of $\U$.
    \item We write $\acl_0$ and $\dcl_0$ for algebraic and definable closures in $T_0$, respectively.
    \item For a set $A$, we define $\rf_0(A)=\rf\cap\dcl_0(A)$ and $\vg_0(A)=\vg\cap\dcl_0(A)$,
    \item For any model $M\models T$, we write $M_0$ for its field-theoretic algebraic closure.
    
    \item Forking independence in $\U$ is denoted by $\forkindep$, and forking independence in $\Ualg$ is denoted by $\forkindep^0$.
    \item For a global type $p \in S(\U)$, let $p_0$ denote its corresponding global type in $\Ualg$; that is if $a \models p$ in $\U$, then $a \models p_0$ in $\Ualg$.
\end{enumerate} 
\end{notation}


\begin{defn}
    Let $K$ be a field. For tuples $a$, $b$, and a parameter set $C$ in $K$, we write $a\forkindep_{C}^{alg} b$ if every finite subtuple $a'$ of $a$ that is algebraically independent over $C$ in the field-theoretic sense remains algebraically independent over $Cb$ in the same sense.
\end{defn}

Recall that we denote the two sorted structure $(\rf,\mathcal{A})$ by $\rf_{\mathcal{A}}$. 

\begin{defn}\label{def_res_dom}
    Let $a$ and $C$ lie in $\U$. The type $\tpp(a/C)$ is\textit{residually dominated} if for every $b\in \U$, we have $\tpp(b/C\rf_{\mathcal{A}}(Ca))\vdash \tpp(b/Ca)$ whenever $\rf(Ca)\forkindep^{alg}_{\rf(C)}\rf(Cb)$.
\end{defn}

By Remark \ref{stabembedded_with_control_of_parameters}, one can show that the conditions $\tpp(b/C\rf_{A}(Ca)) \vdash \tpp(b/Ca)$ and $\tpp(a/C\rf_{A}(Ca))\vdash \tpp(a/Cb)$ are equivalent. 

\begin{notation}
For a valued field $L$, we write $\rf_L:=\{\res(a)$ $|$ $ a\in L\}$ and $\Gamma_L:= \{\val(a)$ $ \mid$ $ a\in L\}$.
\end{notation}
We know recall several algebraic statements regarding separated basis and domination.
\begin{fact}\label{generated_by_rf_and_gamma} (\cite[Proposition 2.7]{EHS23}) Let $L$, $M$, and $C$ be valued fields with $C \subseteq L \cap M$. Assume that: \begin{enumerate} 
\item $L$ or $M$ has the separated basis property over $C$, \item $\Gamma_L\cap \Gamma_M = \Gamma_C$,

\item $\rf_L$ and $\rf_M$ are linearly disjoint over $\rf_C$. 
\end{enumerate} 
Let $N = LM$ be the field generated by $L$ and $M$. Then $L$ and $M$ are linearly disjoint over $C$, $\Gamma_N$ is generated by $\Gamma_L$ and $\Gamma_M$ as groups, and $\rf_N$ is generated by $\rf_L$ and $\rf_M$ as fields. In particular, $\tpp_0(L/C\rf_M \Gamma_M)\vdash \tpp_0(L/M)$.
\end{fact}

The following result is essentially \cite[Theorem 3.8]{Vic22}. In the original statement, the base set is assumed to be a maximal model; here, we weaken this assumption assuming the existence of the separated basis property. For our purposes, we use a stronger assumption on the value group. Aside from these modifications, the proof remains essentially unchanged, but we include it here for the sake of completeness.

\begin{fact}\label{predomination_residual_domination}
    Let $L,M\subset \U$ be substructures and $C\subseteq L\cap M$ be a subfield. Assume that $L$ or $M$ has the separated basis property over $C$. Furthermore, assume $\vg(L)=\vg(C)$ and $\rf_M$ and $\rf_L$ are linearly disjoint over $\rf_C$. Then $\tpp(L/C \rf_{\mathcal{A}}(M) \vg(M))\vdash \tpp(L/M)$.
\end{fact}
\begin{proof}
   Let $L'\models \tpp(C\rf_{\A}(M)\Gamma(M))$ and let $\sigma: LM \to L'M'$ be an $\mathcal{L}$-isomorphism, sending $L$ to $L'$, $M\to M'$ and fixing $C\cup \rf_{\A}(M)\cup\Gamma(M)$. 
   By Fact \ref{generated_by_rf_and_gamma}, there exists a valued field isomorphism $\tau: L'M'\to L'M$, fixing $L'$ and extending $\sigma^{-1}|_{M'}$. Set  $h=\tau\circ \sigma$. Then $h$ is a valued field isomorphism between $LM$ and $L'M$ fixing $M$ and sending $L$ to $L'$. It remains to show that $h$ is an $\mathcal{L}$-isomorphism; that is, it preserves all residue maps $\res_n$.

   Let $a,b\in LM$ with $\res^n(a)=\res^n(b)$. By Fact \ref{generated_by_rf_and_gamma}, there exist $l\in L$ and $m\in M$ with $\val(a)=\val(l)+\val(m)$. By assumption, there is also $c\in \dcl(C)$ with $\val(c)=\val(l)$. 
   Let $l'=lc^{-1}$ and $m'=mc$ and let $x=\frac{a}{l'm'}$. Note that $\val(x)$, $\val(l')$ and $\val(m')$ all lie in $n\Gamma$. Thus, $$\res^n(a)=\pi_n(\res(x))\res^{n}(l')\res^n(m').$$
   Similarly, we can find $y\in \mathcal{O}_{LM}^{\times}$, $l''\in L$ and $m''\in M$ such that $$\res^n(b)=\pi_n(\res(y))\res^n(l'')\res^n(m'').$$
   
By Fact \ref{generated_by_rf_and_gamma}, $\res(x)$ and $\res(y)$ lie in $\rf_L\rf_M$. Moreover, $\res^n(m')$ and $\res^n(m'')$ are elements of $\A(M)$, so we can conclude that $\res^n(a)=\res^n(b)$ can be represented by a formula in $\tpp(L/C\rf_{\A}(M)\Gamma(M))$. 
Since $h$ extends an $\mathcal{L}$-elementary map $\sigma$ that is the identity map on $C\cup \rf_{\A}(M)\cup \Gamma(M)$, we conclude that $\res^n(h(a))=\res^n(h(b))$; therefore $h$ extends to an $\mathcal{L}$-definable map via $\res^n(x)\to \res^n(h(x))$. By Fact \ref{QE_val_A}, $h$ is elementary and can be extended to an $\mathcal{L}$-isomorphism of $\U$.\end{proof}

 This yields the following result, which will be useful later.
 \begin{prop}\label{dcl_in_k} Let $C$ be a subfield and $a$ be a tuple of valued field elements. Let $L=dcl_0(Ca)$ such that $L$ is a regular extension of $C$. Assume that $\tpp_0(a/C)$ is stably dominated. Then for any tuple $b$ from the valued sort with $a \forkindep_M b$ or $b \forkindep_M a$, we have $\rf_0(Mab) = \dcl_0(\rf_0(Ma), \rf_0(Mb))$.

\end{prop}

\begin{proof}

Define $L_1=\dcl_0(Ca)$ and $L_2=\dcl_0(Cb)$. Given that either $a\forkindep_C b$ or $b\forkindep_C a$, we have $\rf_{L_1}\forkindep_{\rf_{C}}^{alg}\rf_{L_2}$. Since $\rf_{L_1}$ is a regular extension of $\rf_{C}$, this implies by Fact \ref{regularity} below that $\rf_{L_1}$ and $\rf_{L_2}$ are linearly disjoint over $\rf_C$. By Fact \ref{algebraic_stabdom}, we have $\Gamma_{L_1}=\Gamma_C$ and that $L_1$ is separated over $M$. Then we can apply Fact \ref{generated_by_rf_and_gamma}; for $N=L_1L_2$, $\rf_{N}$ is generated by $\rf_{L_1}$ and $\rf_{L_2}$ as fields. This is equivalent to $\rf_0(Cab)=\dcl_0(\rf_0(Ca), \rf_0(Cb))$. \end{proof}

Notice that Fact \ref{predomination_residual_domination} does not state that $\tpp(L/C)$ is residually dominated, we also need to consider the parameters outside of the valued field sort. We will show that it is indeed enough to consider the parameters inside the valued field sort. 


\begin{thm}\label{resolution_full} Let $a$ be a tuple in $\U
$ and $C$ be a parameter set in the valued field sort. Then the following are equivalent:
\begin{itemize}
\item[(i)] For any finite tuple $b\in\U$, if $\rf(Ca)\forkindep_{\rf(C)}^{alg}\rf(Cb)$, then $\tpp(a/C\rf_{\mathcal{A}}(Cb))\vdash\tpp(a/Cb)$,
\item[(ii)] For any finite tuple $b\in \vf(\U)$, if $\rf(Ca)\forkindep_{\rf(C)}^{alg}\rf(Cb)$, then $\tpp(a/C\rf_{\mathcal{A}}(Cb))\vdash\tpp(a/Cb)$.
\end{itemize}

\end{thm} 

\begin{proof} The direction $(i) \Rightarrow (ii)$ is obvious. 
For the converse direction, take a tuple $b \in \U$ with $\rf(Ca) \forkindep_{\rf(C)}^{\alg} \rf(Cb)$. 
We choose a model $M$ containing $Cb$ such that $\rf(Ca) \forkindep_{C \rf(Cb)}^{alg} \rf(M)$. By transitivity of $\forkindep^{\alg}$, it follows that $\rf(Ca) \forkindep_{\rf(C)}^{\alg} \rf(M)$. By $(ii)$, we then have $\tpp(a / C \rf_{\mathcal{A}}(M)) \vdash \tpp(a / M)$, which is equivalent to $\tpp(M / \rf_{\mathcal{A}}(Ca)) \vdash \tpp(M / Ca)$. 
It immediately follows that $\tpp(b / C \rf_{\mathcal{A}}(Ca)) \vdash \tpp(b / Ca)$, equivalently $\tpp(a / C \rf_{\mathcal{A}}(Cb)) \vdash \tpp(a / Cb)$.\end{proof}
Finally, using Fact \ref{predomination_residual_domination} and Theorem \ref{resolution_full}, we obtain the following corollary.

\begin{cor}\label{full_residual_domination} Let $a$ be a tuple in the valued field sort, and let $C$ be a subset of the valued field. Assume that $L:=\dcl(Ca)$ is separated over $C$, $\Gamma(L)=\Gamma(C)$, and $\rf(L)$ is a regular extension of $\rf(C)$. Then  $\tpp(a/C)$ is residually dominated. \end{cor}

\subsection{Equivalence of Residual Domination and Stable Domination}

 In this section, our main result is the equivalence between residual domination in the henselian field $\U$ and stable domination in the algebraically closed valued field containing $\U$ as a substructure. Along the way, we also extend the results of Section 2 in \cite{EHS23} to the henselian setting. 
 
 We will need the following results from $\cite{EHS23}$.
 
\begin{fact}\label{type_implications}(\cite[Lemma 1.13]{EHS23}) In any theory, let $C, F$ and $L$ be subsets of a universal model with $C\subseteq L \cap F.$ Then $\tpp(L/C)\vdash\tpp(L/F)$ if and only if $\tpp(F/C)\vdash \tpp(F/L)$.
\end{fact}

\begin{fact}\label{type_implications2}(\cite[Lemma 1.19]{EHS23})
 Let $C\subseteq L$ be valued fields such that $L$ is a regular extension of $C$ and $L$ is henselian. Then $\tpp_0(L/C)\vdash \tpp_0(L/C^{alg})$.
\end{fact}

\begin{fact}\label{regularity}(\cite[VIII, 4.12]{La02})
Let $C\subseteq L$ be fields with $L$ regular over $C$. Then for any field $M$, $L\forkindep_{C}^{alg} M$ implies that $L$ and $M$ are linearly disjoint over $C$. 
\end{fact}
\begin{notation}
    For a subset $A$ of a field $K$, we define $\alg(A):=\mathbb{Q}(A)^{\alg}$, where $\mathbb{Q}(A)$ is the field generated by $A$, and $\mathbb{Q}(A)^{\alg}$ is its field-theoretic algebraic closure.
\end{notation} 
 \begin{lem}\label{linear_disjoint} Let $C,F$ and $L$ be valued fields in $\U$ such that $C\subseteq F\cap L$. Assume that $L$ is a regular extension of $C$ and that $\tpp(L/C)\vdash \tpp(L/F)$. Then $L$ and $F$ are linearly disjoint over $C$.
\end{lem}
\begin{proof}
Suppose on the contrary that $L$ and $F$ are not linearly disjoint over $C$. By Fact \ref{regularity}, there exists a finite tuple $l_1,\cdots, l_n\in L$ which is algebraically independent over $C$ but algebraically dependent over $F$. We may assume $l_1\in \alg(Fl_1,\cdots, l_n)$, and hence $l_1\in \acl(Fl_1, \cdots, l_n)$. Let $X$ be a set of all conjugates of $l_1$ over $F l_1,\cdots, l_n$ inside $\U$. Since $\tpp(L/C)\vdash\tpp(L/F)$, any automorphism fixing $Cl_2,\cdots, l_n$ and sending $l_1$ to some $m$ extends to a one fixing $Fl_2,\cdots, l_n$ and sending $l_1$ to $m$. But, by assumption one can choose $m\not\in X$, a contradiction. 
\end{proof}
 
We now show that if a type is residually dominated, then its realization does not add any new information to the value group.

\begin{lem}\label{initial_orthogonal_to_Gamma}
    Let $a$ be a tuple in $\vf(\U)$ and $C$ be a subfield such that $\tpp(a/C)$ is residually dominated. Then $\Gamma(Ca)=\Gamma(C)$.
\end{lem}

\begin{proof}
    Suppose on the contrary that there exists $\gamma \in \Gamma(Ca) \setminus \Gamma(C)$. First note that $\rf(C\gamma)=\rf(C)$. Let $\alpha\in \rf(C\gamma)$ and let $\varphi(x,\gamma,d)$ be formula witnessing this, where $d$ is a tuple in $\dcl(C)$. 
    By orthogonality of $\rf_{\mathcal{A}}$ to $\Gamma$ and Remark \ref{stabembedded_with_control_of_parameters}, there is a formula $\psi(x,d')$ in the sort $\rf_{\A}$, with $d'\in \rf_{\A}(C)$ such that $\alpha$ is the unique realization. 
    Hence, $\alpha\in \rf(C)$ and thus $\rf(C)=\rf(C\gamma)$.
   Since $\rf(Ca) \forkindep_{\rf(C)}^{alg} \rf(C\gamma)$ holds trivially, residually domination gives $\tpp(\gamma / C \rf_{\mathcal{A}}(Ca)) \vdash \tpp(\gamma / Ca)$. But by orthogonality, if $\gamma \in \dcl(C \rf_{\mathcal{A}}(Ca))$, then $\gamma \in \Gamma(C)$, contradicting the assumption that $\gamma\not\in \Gamma(C)$.\end{proof}

\begin{lem}\label{separated_basis} Let $C\subseteq L$ be subfields of $\U$. Assume that $L$ is a regular extension of $C$ and $\tpp(L/C)$ is residually dominated. Then for any maximal immediate extension $F$ of $C$ in $\U$, we have $\tpp_0(L/C)\vdash \tpp_0(L/F)$ in $\Ualg$.
\end{lem}
\begin{proof} Fix a maximal immediate extension $F$ of $C$. 
Since $\rf(F)=\rf(C)$, $\rf(L)\forkindep_{\rf(C)}^{alg}\rf(F)$ holds trivially. 
By residual domination, this gives $\tpp(L/C\rf_{\A}(F))\vdash \tpp(L/F)$. But, $\rf(F)=\rf(C)$, hence also $\A(C)=\A(F)$; so in fact $\tpp(L/C)\vdash \tpp(L/F)$ holds. 

We now show that every valued field isomorphism $\sigma: L\to L'$ fixing $C$ extends to a valued field isomorphism $\tau: LF\to L'F$, fixing $F$ pointwise. 
By Lemma \ref{linear_disjoint}, $L$ and $F$ are linearly disjoint over $C$. Thus there is a ring isomorphism $\tau: LF\to L'F$ fixing $F$ pointwise and extending $\sigma$.
Moreover, by \cite[Proposition 12.1]{HHM08}, any $F$-span of a finite tuple from $L$ has a separated basis over $F$. So, for every $d\in LF$, there exists a separated $F$-linearly independent tuple $l_1,\cdots, l_n \in L$ and coefficients $f_1,\cdots, f_n \in F$ such that $d=\sum f_i l_i$, and $\val(d)=\min\{ \val(f_i)+\val(l_i)\}$. Since $\Gamma(L)=\Gamma(C)=\Gamma(F)$, it follows that $\val(\tau(d))=\tau(\val(d))$. Thus $\tau$ is a valued field isomorphism etending $\sigma$, showing that $\tpp_0(L/C)\vdash \tpp_0(L/F)$.
\end{proof}

\begin{thm} \label{resdom_implies_separated_basis}
    Let $C$ be a field and $a$ be a tuple from the field sort. Assume that $\tpp(a/C)$ is residually dominated and $L:=\dcl(Ca)$ is a regular extension of $C$. Then $L$ is separated over $C$.
\end{thm}
\begin{proof} 
We first consider two easy cases. If  $a\in \dcl(C)$, then the result is immediate. If $\Gamma(C)$ is trivial, then Lemma \ref{initial_orthogonal_to_Gamma} implies $\Gamma(L)$ is also trivial, and the result follows immediately. Thus, assume $a\not\in \dcl(C)$ and the valuation on $C$ is nontrivial.

Since $L$ is $\dcl$-closed and is regular over $C$, we have $C=\dcl(C)$, and hence $C$ is henselian.
Let $F$ be a maximal immediate extension of $C$ inside $\U$. 
We may further assume that $F$ is an elementary extension of $C$ in the theory $T':=Th(C)$. Note that $T$ and $T'$ agree on  all quantifier free sentences with parameters from $C$.


Since $\tpp(a/C)$ is residually dominated, Lemma \ref{separated_basis} yields $\tpp_0(L/C) \vdash \tpp_0(L/F)$. Then by Lemma \ref{linear_disjoint}, $L$ and F are linearly disjoint over $C$. Since $F$ is maximally complete, by Fact \ref{SepBasisMaximal}, $LF$ is a separated extension of $F$. As $L$ and $F$ are linearly disjoint over $C$, it follows that $\Gamma_{LF}$ is generated by $\Gamma_L = \Gamma(L)$ and $\Gamma_F = \Gamma(F)$. By Lemma \ref{initial_orthogonal_to_Gamma}, we have $\Gamma(L) = \Gamma(C) = \Gamma(F)$. In particular, $\Gamma(LF)=\Gamma(C)$. 

We now prove by induction on $n$ that every $n$-dimensional $C$-vector subspace of $L$ has a separated basis  over $C$. 
Base case, $n = 1$, is immediate. 
Assume that $\vec{l} = (l_1, \dots, l_n) \in L^n$ is $C$-linearly independent and separated over $C$. 
Let $t \in L$ be such that $t$ does not lie in the $C$-span $ C \cdot \vec{l}$ of $\vec{l}$. 
Since $L$ and $F$ are linearly disjoint over $C$, the tuple $(l_1, \dots, l_{n},t)$ is $F$-linearly independent. 
As in the proof of \cite[Proposition 12.1]{HHM08}, there exists $u \in F \cdot \vec{l}$ such that $\val(u) = \max \{ \val(w - t) \text{ : } w \in F \cdot \vec{l}.\}$ Write $\gamma=\val(u)$. Since $\Gamma(LF)=\Gamma(C)$, we have $\gamma\in \Gamma(C)$.

Let $\theta(x_1,\cdots,x_n)\in \tpp_0(L/F)$ be the formula $\val(\sum\limits_{i=1}^{n} x_il_i)=\gamma$. Note that $\theta$ is quantifier-free. 
Since $\tpp_0(L/C)\vdash \tpp_0(L/F)$, by Fact \ref{type_implications}, we have $\tpp_0(F/C)\vdash \tpp_0(F/L)$. It follows that there is $\rho(x_1,\cdots, x_n)\in \tpp_0(F/C)$ that implies $\theta$ (in the theory $T$). As $F$ is an elementary extension of $C$ in the theory $T'$, there is a tuple $d_1,\cdots, d_n\in C$ with $T'\models \rho(d_1,\cdots, d_n)$. Since $\rho$ is quantifier-free, it follows that $T\models \rho(d_1,\cdots, d_n)$, as well. 

For $w=\sum\limits_{i=1}^n d_i l_{i}\in C\cdot \vec{l}$,
and set  $l_{n+1}=t-w$. Then, $l_1,\dots,l_n,l_{n+1}$ is separated $C$. Suppose not. Then there exist $c_1,\dots,c_{n+1}\in C$ such that $\val(\sum\limits_{i=1}^{n+1}c_il_i)>\min\{\val(c_{n+1}l_{n+1}), \val(\sum\limits_{i=1}^nc_il_i)\}$. Then, we have\begin{align*}
    \val(l_{n+1}+\sum\limits_{i=1}^nc_{n+1}^{-1}c_il_i)>\val(l_{n+1}),\\
    \val(t-(w-\sum\limits_{i=1}^nc_{n+1}^{-1}c_il_i))>\gamma
\end{align*}
This contradicts with the choice of $\gamma$. \end{proof}

The following is an analog of Fact \ref{algebraic_stabdom} for henselian fields of equicharacteristic zero. 

\begin{cor}\label{algebraic_resdom} Let $C$ be a subfield and $a$ be a tuple in the valued field sort of $\U$. Assume that $L:=\dcl(Ca)$ is a regular extension of $C$. Then the following are equivalent:

\begin{enumerate}
\item $\tpp(a/C)$ is residually dominated.
\item  $L$ is an unramified extension of $C$ and $L$ has the separated basis property over $C$. 
\end{enumerate}
\end{cor}
\begin{proof}
By Lemma \ref{initial_orthogonal_to_Gamma} and Theorem \ref{resdom_implies_separated_basis}, $1.$ implies $2.$ Conversely, since $L$ is henselian and the extension $L/C$ is regular, it follows that $\rf(L)$ is regular over $\rf(C)$. Then the result follows from Corollary \ref{full_residual_domination}. \end{proof}

Finally, we combine all the results to relate residual domination to stable domination.

\begin{thm}\label{stabdom_eqv_resdom} Let $a$ be a tuple in the valued field sort and $C$ be subfield of $\vf(\U)$. Assume that $C$ is $\acl$-closed in the valued field sort. Then the following are equivalent:
\begin{enumerate}
    \item $\tpp(a/C)$ is residually dominated in the $\mathcal{L}$-structure $\U$,
    \item $\tpp_0(a/C)$ is stably dominated in the $\mathcal{L}_0$-structure $\Ualg$.
\end{enumerate}
\end{thm}

\begin{proof}
    Let $L:= \dcl_0(Ca)$ and $M:=\dcl(Ca)$ First assume that $\tpp(a/C)$ is residually dominated. Since $C$ is $\acl$-closed in $\U$ and $L,M\subseteq \U$, both $L$ and $M$ are regular extensions of $C$. By Corollary \ref{algebraic_resdom}, $M$ has the separated basis property over $C$ and $\Gamma(M)=\Gamma(C)$. As $M\supseteq L$, it follows that $\Gamma_0(L)=\Gamma_0(C)$ and $L$ has the separated basis property over $C$. By Fact \ref{algebraic_stabdom} then, $\tpp_0(a/C)$ is stably dominated.

Conversely, assume that $\tpp_0(a/C)$ is stably dominated. By Fact \ref{algebraic_stabdom}, $L$ is separated over $C$ and $\Gamma_0(L)=\Gamma_0(C)$ in $\Ualg$. It follows that $\Gamma(M)=\Gamma(C)$, and moreover, since $M=\dcl(Ca)$ is an algebraic extension of $L$ and $L$ is henselian, by Fact \ref{Delon2}, $M$ is separated over $L$. Then, by Proposition \ref{transitivity_of_separated_basis}, $M$ is separated over $C$, as well. 
Finally, we apply  Corollary \ref{algebraic_resdom} to conclude that $\tpp(a/C)$ is residually dominated. \end{proof}


\subsection{Properties of Residually Dominated Types}\label{sec_prop_resdom}
In this section, we use Theorem \ref{stabdom_eqv_resdom} to show that residually dominated types share properties similar to those of stably dominated types in $\ACVF$, as listed in \ref{stable domination}.

In \cite[Theorem 3.14]{EHM19}, the characterization of forking in $\ACVF$ over models is given under the assumption that $C$ is maximal. This maximality ensures the existence of the separated basis property. The proof goes through unchanged if we instead assume that the separated basis property holds over $C$.

\begin{prop}\label{forkindep_acvf} Let $C$ be a valued field, and let $a$ and $b$ be tuples from the field sort. Assume that $L := \dcl_0(Ca)$ has the separated basis property over $C$, $\rf_{L}$ is a regular extension of $\rf_C$, and $\Gamma_L/\Gamma_C$ is torsion-free (or $\Gamma_L\cap \Gamma_M=\Gamma_C$, where $M=\dcl_0(Mb)$). Then $a \forkindep^{0}_C b$ if and only if $\rf_0(Ca)\Gamma_0(Ca) \forkindep^{0}_C \rf_0(Cb)\Gamma_0(Cb)$ in the structure $\Ualg$.
    \end{prop}

When the type is residually dominated, the independence in $\Gamma$ comes for free, and (algebraic) independence in the residue field of $\U$ implies (forking) independence in the stable sorts.

\begin{lem}\label{independence_in_rf_implies_independence_in_st}
  Let $C$ be a valued field, with $C=\acl(C)$ in the field sort, and let $a$ and $b$ be tuples of field elements. Suppose that $\tpp(a/C)$ is residually dominated and $\rf(Ca) \forkindep_{\rf(C)}^{\alg} \rf(Cb)$. Then, $a \forkindep^{0}_C b$. In particular, $\st_C(Ca) \forkindep^{0}_C \st_C(Cb)$.
\end{lem}
\begin{proof}
Let $L=\dcl_0(Ca)$ and $M=\dcl_0(Cb)$. 
The residue field in $T_0$ is a model of $\ACF$. Thus, the algebraic independence in the residue field implies $\rf_L \forkindep_{C} \rf_M$ in $\Ualg$.  Since $\tpp(a/C)$ is residually dominated, Corollary \ref{algebraic_resdom} implies that $\Gamma(Ca) = \Gamma(C)$; therefore $\Gamma_0(C) = \Gamma_0(Ca)$ and $\Gamma_L \forkindep_C \Gamma_M$ holds trivially. Because $C$ is $\acl$-closed, $\rf_L$ is a regular extension of $\rf_C$. Applying Proposition \ref{forkindep_acvf}, we conclude that $a \forkindep^{0}_C b$. It immediately follows that $\st_C(Ca) \forkindep^{0}_C \st_C(Cb)$.\end{proof}

Now, we show that residual domination is invariant under base changes. 

\begin{thm}\label{change_of_base}
Let $C \subseteq B$ be subfields that are $\acl$-closed in the valued field sort. Let $p$ be a global type that does not fork over $C$.
\begin{itemize}
    \item[(i)] If $p|C$ is residually dominated, then $p|B$ is also residually dominated.
    \item[(ii)] Suppose further that $p$ is $C$-invariant and $p|B$ is residually dominated. Then $p|C$ is also residually dominated.
\end{itemize}
\end{thm}

 \begin{proof} We start with $(i)$. Let $a \models p | C$. By Theorem \ref{stabdom_eqv_resdom}, $\tpp_0(a/C)$ is stably dominated and by Fact \ref{stabdom_acl_properties} $(ii)$, $\tpp_0(a/\acl_0(C))$ is stably dominated. 
 Let $q$ be the unique $\acl_0(C)$-definable extension of $\tpp_0(a/\acl_0(C))$, given by Fact \ref{stabdom_acl_properties}. 

First, we note that $q$ is $C$-invariant: by Fact \ref{type_implications2}, since  $L = \dcl_0(Ca)$ is henselian and regular over $C$, $\tpp_0(L/C) \vdash \tpp_0(L/\acl_0(C))$. Since $q$ is $\acl_0(C)$-invariant, it follows that it is $C$-invariant.

Let $M \preceq\U$ with $C\subseteq M$, and take $b \models p | M$. Since $p$ does not fork over $C$, we have $\rf(Cb) \forkindep_{\rf(C)}^{\alg} \rf(M)$. By Lemma \ref{independence_in_rf_implies_independence_in_st}, this gives $\st_C(Cb)\forkindep_C\st_C(M)$. By Fact \ref{stabdom_acl_properties} $(i)$, we obtain $b\models q|M$. Since $M$ is arbitrary, we conclude $p_0|\U= q|\U$. Now, let $B\supseteq C$ be $\acl$-closed. By Fact \ref{StabDomExtensions}, $q|B$ is stably dominated. Since $p_0|B=q|B$, applying Theorem \ref{stabdom_eqv_resdom}, we conclude that $p | B$ is residually dominated.

For $(ii)$, assume that $p$ is $C$-invariant and $p|B$ residually dominated. By Theorem \ref{stabdom_eqv_resdom}, we have $p_0|B$ is stably dominated, and by Fact \ref{stabdom_acl_properties}, it extends to a global $\acl_0(B)$-definable type $q$ in $\Ualg$. Similar to above, we aim to apply Fact \ref{StabDomExtensions}, so it suffices to show $q$ is $\acl_0(C)$-invariant. 

By construction, $p_0|B=q|B$. As in part $(i)$, $q|Bc=p_0|Bc$ for any tuple $c$ in $\U$, thus $p_0|\U=q|\U$. By Fact \ref{type_implications2}, $p_0|\U=q|\U$ is, in fact, $\mathcal{L}_0(B)$-definable. Given an $\mathcal{L}_0$-formula $\phi(x;y)$, let $d_{\phi}(y)$ be an $\mathcal{L}_0(B)$-formula that satisfies:
\begin{center}
    for all $m$ in $\Ualg$, $\phi(x;m)\in p_0$ if and only if $\models d_{\phi}(m)$.
\end{center}
By $C$-invariance of $p$ in $\U$, the set $d_{\phi}(\U)$ is $\mathcal{L}(C)$-definable: 
for any $\sigma\in Aut(\U/C)$ and $m\in d_{\phi}(\U)$, $\phi(x;\sigma(m))\in p_0$, hence $\sigma(m)\in d_{\phi}(\U)$.

Let $d$ be a code of $d_{\phi}(\U)$ in $\U^{eq}$, and let $d'$ be a code of $d_\phi(\Ualg)$ in $\Ualg$. Note that we can choose $d'\in\U^{eq}$, because the set it codes is definable over $B\subseteq \U$. Then, we have $d'\in \dcl^{eq}(d)\subset \dcl^{eq}(C)$. It follows that the set $d_{\phi}(\Ualg)$ is $\mathcal{L}_0(\acl_0(C))$-definable, and hence $p_0$ is $\acl_0(C)$-invariant. 

By Fact \ref{StabDomExtensions}, we conclude that $p_0|\acl_0(C)$ is stably dominated. By Facts \ref{stabdom_acl_properties} and \ref{type_implications2}, we have $p_0|C$ is stably dominated as well. Finally, by Theorem \ref{stabdom_eqv_resdom}, $p|C$ is residually dominated, as desired. \end{proof}

Using this theorem, we show that residually dominated types are orthogonal to $\Gamma$ in the following sense.

\begin{thm}\label{orthogonality_resdom}
    Let $C$ be a valued field such that $C$ is $\acl$-closed in the valued field sort and assume that $p$ is a $C$-invariant type. Then $p|C$ is residually dominated if and only if for any model $M\supseteq C$ and $a\models p|M$, we have $\Gamma(Ma)=\Gamma(M)$.
\end{thm}
\begin{proof}

First, assume that $p|C$ is residually dominated. Let $M$ be a model containing $C$. Since $M$ is $\acl$-closed, it follows by Theorem \ref{change_of_base} $(i)$ that $p|M$ is residually dominated. Consequently, by Corollary \ref{algebraic_resdom}, for $a \models p|M$, we have $\Gamma(M) = \Gamma(Ma)$.

For the converse, fix a maximally complete model $M\supseteq C$, and let $a \models p|M$. By assumption, we have $\Gamma(Ma)=\Gamma(M)$. By Fact \ref{SepBasisMaximal}, the field $L = \dcl(Ma)$ has the separated basis property over $M$. Since $M$ is a $\acl$-closed, $L$ is a regular extension of $M$. Thus, applying Corollary \ref{algebraic_resdom}, we conclude that $\tpp(a/M)$ is residually dominated. Finally, as $p$ is $C$-invariant, Theorem \ref{change_of_base} $(ii)$ yields that $p | C$ is residually dominated. \end{proof}

\section{Residually Dominated Groups}\label{Residually dominated groups}

In this section, we introduce groups controlled by the residue field, analogous to how stably dominated groups in \cite{HR19} are controlled by their stable part. Recall that we use the Notation \ref{notation_main_2}. 






In \cite{HR19}, a definable group $G$ is said to be \emph{stably dominated} if it has a stably dominated, definable generic. Here, we adapt this definition to arbitrary pure henselian valued fields as follows.

\begin{defn} A definable group $G$ in $\U$ is \textit{residually dominated} if there exists a small model $M$ such that $G$ admits a strongly $f$-generic type over $M$, where this type is residually dominated over $M$.


\end{defn}

\begin{exmp}\begin{enumerate}
\item Any definably amenable group in the residue field, whose theory is $\NTP$, is residually dominated. This follows from the fact that definably amenable groups in an $\NTP$ theory have strongly $f$-generics, and any type in the residue field is residually dominated. 
\item Assume that $\U$ is a real closed-valued field in which case $\A_n=\{1,-1\}$ for even $n$ and $\A_n=\{1\}$ for odd $n$. Let $G = \{ g\in \mathcal{O} $ :  $ \val(g)=0$ and $g>0\}$, the multiplicative group of the positive elements of the valuation ring $\mathcal{O}$. Then there are exactly two residually dominated generics of $G$, namely $p_{\infty}$ and $p_{0^+}$ defined as follows:
    \begin{enumerate}
        \item $p_{\infty}$ is defined as for all set $C$ and $a\models p_{\infty}|C$, $\res(a)\not\in \acl(C)$ and $\res(a)$ is bigger than any element of $\rf(\acl(C))$,
        \item $p_{0^{+}}$ is defined as for all sets $C$ and $a\models p_{0^{+}}|C$, $\res(a)\not\in \acl(C)$, $\res(a)>0$, and it is infinitesimally close to $0$.
        
    \end{enumerate}
    \end{enumerate}
\end{exmp}

As shown in \cite{HR19}, stable domination is witnessed by a group homomorphism. 

\begin{fact}\label{homomorphismACVF}
(Proposition 4.6 and Lemma 4.9, \cite{HR19})
  \textit{Let $G$ be a pro-$C$-definable, stably dominated group over $C$ in the structure $\Ualg$. Then:}
    
    \begin{enumerate}
\item \textit{There exists a pro-$C$-definable group $\mathfrak{g}$ in $\st_C$ and a pro-$C$-definable surjective group homomorphism $\theta: G\to\mathfrak{g}$ such that every generic of $G$ is dominated by $\theta.$ Moreover, the pair $(\mathfrak{g},\theta)$ is universal, i.e. if $\mathfrak{g}'$ is a pro-$C$-definable group in $\st_C$ and  $\theta': G\to \mathfrak{g}'$ is a pro-$C$-definable map over $C$ then there exist a unique homomorphism $\tau:\mathfrak{g}\to \mathfrak{g}'$ with $\tau \circ \theta=\theta'$. }

\item \textit{If $\mathfrak{g}$ is pro-$C$-definable in $\st_C$, and $\theta: G \rightarrow \mathfrak{g}$ is a surjective pro-$C$-definable group homomorphism which dominates $G$, then, any type $p \in S_G(C)$ is generic in $G$ if and only if its image $\theta_*(p)$ is generic in $\mathfrak{g}$.}
    \end{enumerate}
\end{fact}


    \begin{lem}\label{lemma_fin_sat_implies_stationary}
        In any theory, let $q \in S(A)$ be stably dominated, where $A$ is any set of parameters. Assume that $q$ is finitely satisfiable in $A$. Then, $q$ admits a global $A$-definable extension, and is stationary.
    \end{lem}

    \begin{proof}
        Let $B = \acl(A)$. Then, by compactness, the type $q$ admits an extension $q' \in S(B)$ which is finitely satisfiable in $A$. Indeed, the partial type 
        $q(x) \cup \lbrace \neg \phi(x) \, | \, \phi $ is over $B$ and has no points in $A \rbrace$ is consistent, otherwise $q$ would not be finitely satisfiable in $A$.

        On the other hand, the type $q'$ is stably dominated and stationary, because $B = \acl(A)$. It is also finitely satisfiable in $A$. Now, let $M$ be a sufficiently saturated model containing $B$. Let $q''$ denote the unique nonforking extension of $q'$ to $M$. Since some nonforking extension of $q'$ is finitely satisfiable in $A$, uniqueness implies that $q''$ is finitely satisfiable in $A$, thus $A$-invariant. By $A$-invariance, the type $q''$ is then generically stable over $A$. So, it is $A$-definable, and, by \cite[Proposition 2.1]{pillay-tanovic}, the restriction $q''|_A = q$ is stationary.   
    \end{proof}



Theorem \ref{HomomorphismResidue} below is an analogue to Fact \ref{homomorphismACVF} $(i)$. The construction of the pro-definable group homomorphism $\theta$ is nearly identical; however, instead of using the existence of strong germs of definable functions on stably dominated types, we apply Proposition \ref{dcl_in_k}, which allows us to remain within the residue field of $\U$. In the proof, we will also use the following fact.

\begin{fact}\label{nonforking_reduct}(\cite[Lemme 2.1]{BMW11}) Let $\mathcal{L}$ be any language, and let $\mathcal{L}'$ be a reduct of $\mathcal{L}$. Let $T$ and $T'$ be $\mathcal{L}$- and $\mathcal{L}'$-theories, respectively, with $T'$ stable. If $C$ is $\acl$-closed in $T$, and $a$ does not fork with $b$ over $C$ in $T$, then $a$ does not fork with $b$ over $C$ in $T'$.\end{fact}

\begin{thm}\label{HomomorphismResidue}
 Assume that $T$ is $\NTP$. Let $(G,\cdot)$ be a definable residually dominated group in the valued field sort of $\U$, with a strongly $f$ generic over $M$. Then, there exist an $\mathcal{L}_0(M)$-definable group $\mathfrak{g}$ in the residue field and a $\mathcal{L}(M)$-definable surjective group homomorphism $f: G_M^{00}\to \mathfrak{g}$ such that the strongly $f$-generics of $G_M^{00}$ are dominated via $f$. Namely, for each strongly $f$-generic $p$ of $G_M^{00}$ and tuples $a, b$ in $\U$ with $a \models p|M$, we have $\tpp(b/M f(a))\vdash \tpp(b/Ma)$ whenever $f(a)\forkindep_{\rf(M)}^{alg}\rf(Mb)$.
\end{thm}

\begin{proof}
\setcounter{claim}{0}
   Let $H$ be an $M$-definable algebraic group and $\iota:G_M^{00}\to H$ be an $M$-definable group homomorphism with finite kernel, as in Fact \ref{group_chunk_ntp2}. 

  Let $p \in S(\U)$ be a strongly $f$-generic of $G_M^{00}$, 
  and let $p_0$ be the corresponding $\mathcal{L}_0$-type in $\Ualg$. Then $\iota_* (p)$ concentrates on $H$, as $H$ is $\mathcal{L}_0(M)$-definable. It follows that $\iota_* (p_0)$ concentrates on $H$. By Theorem \ref{stabdom_eqv_resdom}, $p_0|M$ is stably dominated, and since $M$ is a model of $T$ and a substructure of $M_0$, the type $p_0|M$ is finitely satisfiable in $M$. Thus, by Lemma \ref{lemma_fin_sat_implies_stationary}, it is stationary, namely $p_0$ is its unique non-forking extension of $p_0|M$. 
  
  Let $q=\iota_{*} (p)$. Note that the corresponding $\mathcal{L}_0$-type $q_0$ equals to $\iota_{*} (p_0)$. Since $p$ is strongly f-generic in $G$, we know that $q$ is strongly f-generic in $\iota(G) \leq H$. Moreover, since $p \in G^{00}_M$, we have $q \in \iota(G)^{00}_M$. We will denote the multiplication of $H$ by $\cdot_H$. We also note that $\cdot_H$ is $\mathcal{L}_0$-definable.
    \begin{claim}\label{claim1_homres} $q_0$ concentrates on $G_1:=Stab_H(q_0)$.
 \end{claim}
    \begin{proof}
First, note that $Stab_H(q_0)$ is $\mathcal{L}_0(M)$-type-definable, thus it suffices to find a realization $b\models q_0|M$ that lies in $ Stab_H(q_0)$. For this, we will find independent realizations $b\models q_0|M$, $a\models q_0|Mb$ such that $b\cdot_{H} a \models q_0|Mb$.

Let $\mu$ be the ideal of definable subsets of $H$ that do not extend to a strongly $f$-generic over $M$. Then, $q$, being strongly $f$-generic, is not in $\mu$. 
In fact, by Fact \ref{genericity} $(ii)$, since $\iota(G)^{00}_M \setminus St(q)$ is contained in a union of $M$-definable sets in $\mu$, we know that $q$ concentrates on $St(q)$. 
So, let $b$ be a realization of $q |M$. 
Then, since $b \in St(q)$, the partial type over $Mb$, given by $q |M \cap b \cdot q|M$, is not in $\mu$. So, it extends into a global type $q' \in S(\U)$, \emph{such that $q'$ is generic over $M$}. Then, let $a$ realize $q'$, in particular $a \models q|M$. Now, by construction, both $\tpp(a/M)$ and $\tpp(b \cdot_H a/M)$ are equal to $q|M$. 
Also, by strong $f$-genericity of $q'$, we have $a \forkindep_M b$ and $b \cdot_H a \forkindep_M b$. 

We observe that these independence relations also hold in $\Ualg$. Let $a\models p|M$. Since $\iota(a)\in \acl_0(Ma)$ and $\tpp_0(a/M)$ is stably dominated, we have $q_0=f_{*}(p_0)$ is stably dominated. By Theorem \ref{stabdom_eqv_resdom}, $q$ is residually dominated. Then, Lemma \ref{independence_in_rf_implies_independence_in_st} implies $b\cdot_H a \forkindep_M^0 b$ and $a\forkindep_M^0 b$ in $\Ualg$. By stationarity of $q_0|M$, we conclude $b\cdot_H a \models q_0|Mb$ and $a \models q_0|Mb$, as desired. \end{proof}


By Fact \ref{genericity} $(ii)$, the group $G_M^{00}$ is generated by the set of realizations of $p|M$ and $p^{-1}|M$. Then, $\iota(G_M^{00})$ is generated by $q|M$ and $q^{-1}|M$. Since $q_0$ is the unique generic of $G_1$, we must have $q_0=(q_0)^{-1}$. Then, by above claim, $\iota(G_M^{00}) \leq G_1$.
\begin{claim}\label{claim2_homres}
        Each strongly $f$-generic of $G_M^{00}$ is residually dominated.
\end{claim}

   \begin{proof} Let $p,q,q_0$ and $G_1$ be as above. Fix a strongly $f$-generic type $r$ over $M$, and let $s := \iota_*(r)$. 
   Let $\mathfrak{g}$ be a pro-$\mathcal{L}_0$-definable group in the stable sorts, and let $\theta: G_1 \to \mathfrak{g}$ be the surjective $\mathcal{L}_0$-definable group homomorphism as in Fact \ref{homomorphismACVF} (1). We wish to show that $s|M$ is residually dominated.
   
     By the previous paragraph, we have $\iota(G_M^{00}) \leq G_1$. In particular, both $s$ and $s_0$ concentrate on $G_1$. Choose $b \in \iota(G_M^{00})(\U)$ such that $b \models q|M$, and let $a \in \iota(G_M^{00})(\U)$ realize $s|Mb$. Then $a \forkindep_M b$. Since $s$ is a strongly $f$-generic of $\iota(G_M^{00})$, we also have $b \cdot_H a \forkindep_M b$. 
   
   Note that $\mathfrak{g}$ is also definable in $\U$, since all parameters are in $\U$. 
   Then we have $\theta(b \cdot a) \forkindep_M \theta(b)$ in $\U$. 
   Because the reduct is stable, by Fact \ref{nonforking_reduct}, the non-forking independence also holds in $\Ualg$, and by symmetry in stable theories, we also have $\theta(b) \forkindep_M^0 \theta(b \cdot_H a)$, where $\tpp_0(\theta(b)/M)$ is stationary and extends to  the unique generic of $\mathfrak{g}$. It follows from Fact \ref{homomorphismACVF} $(2)$ that $b \models q_0|M(b \cdot_H a)$, and in particular $b \models q_0|M(b \cdot_H a)^{-1}$. By the genericity of $q_0$, the element $(a \cdot_H b)^{-1} \cdot_H b = a^{-1}$ realizes $q_0|M$. Since $q_0|M$ is stably dominated and $a^{-1} \in \acl_0(Ma)$, it follows that $\tpp_0(a/M) = s_0|M$ is stably dominated. Hence, by Theorem \ref{stabdom_eqv_resdom}, the type $s|M=\iota_*(r|M)$ is residually dominated. Since $\iota$ has finite kernel, the type $r|M$ is residually dominated as well. This concludes the proof.\end{proof}


Now, we will construct the desired group homomorphism. For a tuple $a$ in the valued field sort, let $\theta(a)$ be an enumeration of $\rf_0(Ma)$. Let  $a\models q_0|M$, $b\models q_0|Ma$ and $c = a \cdot_H b$. Then, $\theta(c) \in \rf_0(Mab)$. Since $q_0|M$ is stably dominated, by Proposition \ref{dcl_in_k}, $\theta(c) \in \dcl_0(\theta(a), \theta(b))$. Then, there exists an $\mathcal{L}_0(M)$-definable map $F$ such that $F(\theta(a),\theta(b))=\theta(c)$ for all independent realizations of $q_0|M$.

Define $\pi = \theta_{*}(q_0)$. 
By Claim \ref{claim1_homres}, for such $a,b$ and $c$ as above, $c\models q_0|Ma$. By Proposition \ref{dcl_in_k} again, we have $\theta(b)\in \dcl_0(\theta(a), \theta(c))$. Similarly, we can show that $\theta(a)\in \dcl_0(\theta(b),\theta(c))$. 
Then, $(F,
\pi)$ is a group chunk and by \cite[Proposition 3.15]{HR19}, one obtains a pro-$\mathcal{L}_0(M)$-definable group $\mathfrak{g}$ that lies in the residue field, the sort where the image of $F$ lies. 
By \cite[Proposition 3.16]{HR19} $\theta$ extends to a pro-$\mathcal{L}_0(M)$-definable map from $G_1$ to $\mathfrak{g}$, which we again denote it by $\theta$.

Let $f=\theta\circ \iota$.

\begin{claim}\label{claim2_homres}
        Each strongly $f$-generic of $G_M^{00}$ is residually dominated \emph{via $f$.} 
\end{claim}
\begin{proof}
    Fix a strongly $f$-generic $r$ of $G_M^{00}$, and let $b$ be a tuple such that $f(a)\forkindep_{\rf(M)}^{alg}\rf(Mb)$. Assume that $b\equiv_{Mf(a)} b'$. We will show that $b\equiv_{Ma}b'$. 
   
  By Theorem \ref{resolution_full}, we may assume that $b$ is a tuple in the valued field. We will first show that we have $b \equiv_{M\rf_{\A}(M\iota(a))} b'$. 
  In fact, by residual domination of $\tpp(\iota(a)/M)$, we have $\Gamma(M\iota(a)) = \Gamma(M)$. 
Then, for any $x \in \dcl(M\iota(a))$ with $\val(x) \in n\Gamma$, since $M$ is a model, there exists $b \in M$ such that $n\val(b) = \val(x)$. 
  Hence, $\res^n(x) = \pi_n\left(\res\left(\frac{x}{b^n}\right)\right) \in \rf(M\iota(a)) 
  .$ It follows that $\A(M\iota(a))\subset \dcl(\rf(M\iota(a))$.

As $\tpp(\iota(a)/M)$ is residually dominated and $\rf(M\iota(a))\forkindep_M^{alg}\rf(Mb')$, it follows that $b \equiv_{M\iota(a)} b'$. Therefore, it suffices to show that $b \equiv_{Ma} b'$.

   As $\iota$ is an $M$-definable, finite-to-one group homomorphism, we have $Ker(\iota)\subset M$. Let $a=a_1, \cdots, a_n$ be preimages of $\iota(a)$. Then, $a_i=m_ia$ for some $m_i\in Ker(\iota)\subset M$. In particular, $a\in \dcl(M\iota(a))$; $a\in \dcl((\sum_{i}m_i)\cdot a,M)\subset \dcl(M\iota(a))$. Thus, $\tpp(b/M\iota(a))\models \tpp(b/Ma)$, and $b\equiv_{Ma}b'$ follow immediately. \end{proof}

We now continue by following the arguments of the proof of \cite[Proposition 4.6]{HR19}. By \cite[Proposition 3.4]{HR19}, both $G_1$ and $\mathfrak{g}$ are pro-definable limit $\mathcal{L}_0(M)$-definable groups. Since $\mathfrak{g}$ is definable in a stable theory, as in the proof of \cite[Proposition 4.6]{HR19} $\theta$ is surjective; we write $\mathfrak{g}_i=\theta_i(G_1)$. By the same proof, we also know that there exists $i_0$ such that $\theta_i (b)\in \acl_0(M\theta_{i_0}(b))$ for all $i\geq i_0$, and $b\models q_0|M$, where $q_0$ is the  unique generic of $G_1$.

 Let $\tilde{f}=\theta_0\circ \iota: G_M^{00}\to \theta_{i_0}(G_1)$ and let $p$ be a strongly $f$- generic of $G_M^{00}$. We will show that $p$ is residually dominated via $\tilde{f}$. For $a\models p|M$ and $b$ with $\rf(Mb)\forkindep_M^{alg}\tilde{f}(a)$, since $f(a)\in \acl(\tilde{f}(a))$, we have $\rf(Mb)\forkindep^{alg}_M f(a)$. By Claim \ref{claim2_homres}, it suffices to show $f(a)\in \dcl(\tilde{f}(a))$, so that we have: $\tpp(b/M\tilde{f}(a))\vdash\tpp(b/Mf(a))\vdash\tpp(b/Ma).$
 
 First, we will show that the $\mathcal{L}(M)$-definable group homomorphisms $\theta_{k,i_0}:\theta_k(G_1)\to \theta_{i_0}(G_1)$ that define the pro-definable group $\theta(G_1)$ are finite-to-one for all $k\geq i_0$. 
 For $c,d \in G_1$ with $d$ realizes the generic $q_0|M$, assume that $\theta_{k,i_0}(\theta_k(c)) = \theta_{i_0}(d)$. 
Then $\theta_k(d) \in \acl(\theta_k(c))$, so $\dim(\theta_k(d)/M) = \dim(\theta_k(c)/M)$; in particular,
$\theta_k(c)$ is a generic of $\theta_k(G_1)$. By surjectivity of $\theta_k$, we may assume that $c$ is a generic of $G_1$. Then, the set of preimages of $\theta_k(a)$ is finite, since for any generic $c$, we have $c\equiv_M a$ and $\theta_{k}(c)\in \acl(M\theta_{i_0}(a))$. 

 Now, assume that $\theta_{i_0}(c)\in\theta_{i_0}( G_1)$ is not generic. 
 There is $d\models q_0|M$ with $c\cdot d\models q_0|M$; thus $\theta_{i_0}(c\cdot d)$ is a generic of $\theta_{i_0}(G)$. There are finitely many elements in $\theta_k(G_1)$ sent to $\theta_{i_0}(c\cdot d)=\theta_{i_0}(c)\cdot\theta_{i_0}(d)$ and finitely many sent to $\theta_{i_0}(d)$, so the same must hold for $\theta_{i_0}(c)$. 
 This shows that $\theta_{k,i_0}$ is a finite-to-one $M$-definable group homomorphism, hence its kernel lies in $M$. As discussed in the last part of the proof Claim \ref{claim2_homres}, we have $\theta_k(c)\in \dcl(M\theta_{i_0}(c))$ for all $k\geq i_0$ and $c\models q_0|M$. It immediately follows that $f(a)\in \dcl(M\tilde{f}(a))$, for all $a\models p|M$, as requested. \end{proof}

In fact, the proof of the theorem yields another key result:

\begin{thm}\label{thm_acvf_enveloppe} Assume $T$ is $\NIP$.
    Let $(G,\cdot)$ be a residually dominated group lying in the valued field sort of $\U$. Then, there exists an $\mathcal{L}_0$-type-definable group $G_1$, a definable group homomorphism $\iota: G^{00} \rightarrow G_1$ with finite kernel, with the following properties:

    \begin{enumerate}
        \item The group $G_1$ is, in the structure $\Ualg$, connected stably dominated, and the image of any generic of $G^{00}$ is generic, in the sense of $ACVF$, in $G_1$.
        
        \item For all $\mathcal{L}_0$-type-definable groups $H_1$, for all definable morphisms $f: G^{00} \rightarrow H_1$, there exist $\mathcal{L}_0$-type-definable $H'_1 \leq H_1$, $H_0 \lhd H'_1$ finite, and $\psi: G_1 \rightarrow H'_1/H_0$ such that $f$ factors through $H'_1$ and $\pi \circ f = \psi \circ \iota$, where $\pi : H'_1 \rightarrow H'_1 / H_0$ is the quotient map. In other words, the following diagram commutes:
    \end{enumerate}

\begin{center}
\begin{tikzcd} G^{00}\arrow[d, "f"] \arrow[r, "\iota"]& G_1 \arrow[d, "\psi"] \\
 H'_1\arrow[r, "\pi"]& H'_1 / H_0  \end{tikzcd}
\end{center}

    Moreover, such a group $G_1$ is unique up to $\mathcal{L}_0$-definable isogeny. 
\end{thm}

\begin{proof}
\setcounter{claim}{0}
    
   Let $M$ be a small model such that $G$ is definable over $M$ and the generics in $G^{00}$ are dominated over $M$. 
    Fix some $p \in S(M)$, residually dominated generic of $G^{00}$. Let $H$ and $\iota: G\to H$ be as in the proof of the previous theorem. Let $q_0$ be the $\ACVF$-part of $q:=\iota_*(p)$. Then, $\iota: G^{00}\to G_1$ is a well-defined map where $G_1=Stab_H(q_0)$ is connected, stably dominated group. 
    



    Let us now prove that all generics of $G^{00}$ are sent to the $\ACVF$-generic of $G_1$. We use Fact \ref{homomorphismACVF}. Let $\theta: G_1 \rightarrow \mathfrak{g}$ be a dominating surjective homomorphism given by Fact \ref{homomorphismACVF} (1), where $\mathfrak{g}$ is an algebraic group over the residue field $\rf(M)$. 
    
    Let $p_1 \in S(M)$ be some strong f-generic of $G^{00}$. Let $a$ realize $p|M$, and $a_1$ realize $p_1|{M a}$. Then, by genericity of $p_1$, the product $a \cdot a_1$ realizes $p_1|{M a}$. 

    \begin{claim}
        The element $\theta \circ \iota (a_1)$ is generic over $M$ in the algebraic group $\mathfrak{g}$.
    \end{claim}

    \begin{proof}
        We compute algebraic ranks in $\mathfrak{g}$.  To simplify notations, let $\alpha = \theta(\iota(a))$ and $\alpha_1=~\theta(\iota(a_1))$. We know that $a \cdot a_1 \forkindep_M a$, in particular $\alpha \cdot \alpha_1$ is algebraically independent over $k(M)$ from $\alpha$. Thus, we have $\dim(\alpha \cdot \alpha_1, \alpha / k(M)) = \dim(\alpha \cdot \alpha_1/ k(M)) + \dim(\alpha / k(M))$. On the other hand, we have $\dim(\alpha \cdot \alpha_1, \alpha / k(M)) = \dim(\alpha \cdot \alpha_1, \alpha_1 / k(M)) \leq \dim(\alpha \cdot \alpha_1/ k(M)) + \dim(\alpha_1 / k(M))$. Hence, we have $\dim(\alpha_1 / k(M)) \geq \dim(\alpha / k(M))$. Then, recall that $q = \tpp(\iota(a)/M)$ is generic in $G_1$. So, by Fact \ref{homomorphismACVF} (2), the rank $\dim(\alpha / k(M))$ is maximal in $\mathfrak{g}$. Thus, by the inequality proved above, the rank $\dim(\alpha_1 / k(M))$ is also maximal in $\mathfrak{g}$. Since $\mathfrak{g}$ is an algebraic group, this precisely means that $\alpha_1$ is generic over $k(M)$. 
       
    \end{proof}

In other words, the type $(\theta \circ \iota)_*(p_1)$ is generic in $\mathfrak{g}$. So, by Fact \ref{homomorphismACVF} (2), the type $\iota_*(p_1)$ is generic in $G_1$, as required.
    
Now, let $H_1$ be some $\mathcal{L}_0$-type-definable group, and $f: G^{00} \rightarrow H_1$ be a definable morphism. Consider the type $r= f_*(p)$. Then $r$ is a residually dominated type, so $r_0$ is stably dominated. Let $N$ be a sufficiently saturated model over which everything is defined. Let $a$ realize $p|N$, and let $b$ realize $p|{Na}$. 
Since $Ker(h)$ is finite, we have $f(a) \in \acl_0(Na) = \acl_0(N\iota(a))$. Similarly, we have $f(b) \in \acl_0(N\iota(b))$ and $f(b \cdot a) \in \acl_0(N \iota(b \cdot a))$.

\begin{claim}
    There exist $\mathcal{L}_0$-type-definable groups $H'_1 \leq H_1$, $H_0 \lhd H'_1$ finite, and $\psi: G_1 \rightarrow H'_1 / H_0$, over $N$, such that, for all $\alpha$, $\alpha^{'}$ realizing $p|_N$, the element $f(\alpha^{'} \cdot \alpha^{-1})$ is in $H_1$, and $\psi$ sends $\iota(\alpha^{'} \cdot \alpha^{-1})$ to the coset of $f(\alpha^{'} \cdot \alpha^{-1})$.
\end{claim}
\begin{proof}
    Let $t$ denote the invariant
    residually dominated type $\tpp(f(a), \iota(a) / N)$. Then $t_0$ is stably dominated, and finitely satisfiable in $N$. Thus, by Lemma \ref{lemma_fin_sat_implies_stationary}, the type $t_0$ is stationary. 

    Then, Proposition \ref{prop_stabilizer_definable_types} yields $\mathcal{L}_0$-type-definable groups $H'_1 \leq H_1$, $H_0 \lhd H'_1$ finite, and a definable group homomorphism $\psi: G_1 \rightarrow H'_1 / H_0$, over $N$, such that, for all $(a_1, b_1)$, $(a_2, b_2)$ realizing $t_0|_N$, the morphism $\psi$ sends $a_2 \cdot a_1^{-1}$ to the coset of $b_2\cdot b_1^{-1}$. To conclude, it suffices to recall that $t_0$ is the $\mathcal{L}_0$-reduct of $\tpp(f(a), \iota(a) / N)$, and that $f$, $\iota$ and $\psi$ are group homomorphisms.
    \end{proof}


\begin{claim}
    For all $\alpha, \alpha'$ such that $\alpha$ realizes $p|_N$ and $\alpha'$ realizes $p|_{N\alpha}$, the type $\tpp(\alpha'\cdot \alpha' / N)$ is generic in $G^{00}$.
\end{claim}

\begin{proof}
This is Lemma \ref{lemma_indep_products_of_generics}.
\end{proof}

So, the $\mathcal{L}(N)$-definable morphisms $\psi \circ \iota$ and $\pi \circ f$ coincide on a type in $G^{00}$ which is generic. Thus, they coincide on the subgroup of $G^{00}$ generated by the generic type $p'|N$, so they are equal.

Note that, since $\psi \circ \iota = \pi \circ f$, and the morphisms $\iota$ and $\pi$ have finite kernel, the morphism $\psi$ has finite kernel if and only if $f$ has finite kernel.

\begin{claim}
    The group $G_1$ is unique up to $\mathcal{L}_0$-definable isogeny.
\end{claim}
\begin{proof}
    Let $G_2$ be an $\mathcal{L}_0$-type-definable group which is also connected and stably dominated. Let $\iota_2 : G^{00} \rightarrow G_2$ be a definable homorphism with finite kernel, such that any generic of $G^{00}$ is sent to the generic of $G_2$. Let $N$ be a sufficiently saturated model over which everything is defined. Let $a$ realize $p|N$, let $b$ realize $p|Na$, and let $c$ realize $p|Nab$. Then, the sextuples $(\iota(a), \iota(b), \iota(a\cdot b), \iota(c), \iota(a\cdot b \cdot c), \iota(b \cdot c))$ and $(\iota_2(a), \iota_2(b), \iota_2(a\cdot b), \iota_2(c), \iota_2(a\cdot b \cdot c), \iota_2(b \cdot c))$ are equivalent over $N$, since, by the finite kernel assumption, they are both equivalent over $N$ to $(a, b, a\cdot b, c, a\cdot b \cdot c, b \cdot c)$. Also, from the hypothesis that $\iota$ and $\iota_2$ send $p$ to the generics of $G_1$ and $G_2$ respectively, these sextuples satisfy the hypotheses of Proposition 3.24 of \cite{Wang_group_config}. Thus, applying this proposition, we find an $\mathcal{L}_0$-definable isogeny between $G_1$ and $G_2$, as required. \end{proof}

This concludes the proof of the theorem. \end{proof}

We conclude by proving the functoriality of the construction from Theorem \ref{thm_acvf_enveloppe}.

\begin{prop}\label{prop_functoriality} Assume $T$ is $\NIP$. Let $(G, \cdot)$, $(H, \cdot)$ be residually dominated definable groups, in the valued field sort. Let $f: G \twoheadrightarrow H$ be a surjective definable group homomorphism. Let $G_1$, $H_1$ be $\mathcal{L}_0$-definable groups, and $\iota_G: G^{00} \rightarrow G_1$, $\iota_H: H^{00} \rightarrow H_1$ be definable group homomorphisms satisfying the conclusions of Theorem \ref{thm_acvf_enveloppe}.

Then, there exists a finite normal subgroup $H_0 \lhd H_1$ and an $\mathcal{L}_0$-definable surjective (in ACVF) group homomorphism $\psi: G_1 \twoheadrightarrow H_1 / H_0$ such that the following diagram commutes:

\begin{center}
\begin{tikzcd} G^{00}\arrow[d, "f"] \arrow[rr, "\iota_G"]&& G_1 \arrow[d, "\psi"] \\
 H^{00} \arrow[r, "\iota_H"]& H_1 \arrow[r] & H_1 / H_0  \end{tikzcd}
\end{center}

\end{prop}

\begin{proof}
\setcounter{claim}{0}
    Let $M$ be a sufficiently saturated model over which everything is defined. Let $p \in S_G(M)$ be a residually dominated generic concentrating on $G^{00}$.
    \begin{claim}
        The type $f_*(p) \in S_H(M)$ is a residually dominated generic of $H$ concentrating on $H^{00}$.
    \end{claim}
    \begin{proof}
        By $\NIP$, we know that $Stab_G(p)=G^{00}$, so $p$ concentrates on $Stab_G(p)$. Then, since $f$ is a group homomorphism, this implies that $f_*(p)$ concentrates on $Stab_H(f_*(p))$. Now, by surjectivity of $f$, and genericity of $p$, the type $f_*(p)$ is generic in $H$, so $Stab_H(f_*(p))=H^{00}$. Thus, the type $f_*(p)$ concentrates on $H^{00}$.

        To conclude, recall that residual domination is preserved under pushforward by definable maps. So $f_*(p)$ is residually dominated.
    \end{proof}

    Then, we apply the second point of Theorem \ref{thm_acvf_enveloppe}, to the composite $\iota_H \circ f: G^{00} \rightarrow H^{00} \rightarrow H_1$. So, there exists $\mathcal{L}_0$-definable subgroups $H'_1 \leq H_1$, $H_0 \lhd H'_1$ finite, and an $\mathcal{L}_0$-definable morphism $\psi: G_1 \rightarrow H'_1/H_0$ such that $\iota_H \circ f$ factors through $H'_1$ and $\pi \circ \iota_H \circ f = \psi \circ \iota_G$. In other words, the following diagram commutes:

\begin{center}
\begin{tikzcd} G^{00} \arrow[rd, "\iota_H \circ f"] \arrow[rr, "\iota_G"]&& G_1 \arrow[d, "\psi"] \\
 & H'_1 \arrow[r] & H'_1 / H_0  \end{tikzcd}
\end{center}

\begin{claim}
    The subgroup $H'_1 \leq H_1$ is equal to $H_1$.
\end{claim}
\begin{proof}
    Recall that the type $f_*(p)$ is residually dominated generic of $H^{00}$. Thus, by the first point of Theorem \ref{thm_acvf_enveloppe} applied to $H$, the $\mathcal{L}_0$-restriction of the type ${(\iota_H \circ f)}_*(p)$ is the unique generic of $H_1$ in $\Ualg$. Since $\iota_H \circ f$ factors through $H'_1$, this implies that ${(\iota_H \circ f)}_*(p)$ concentrates on $H'_1$. Since this type is generic in $H_1$, it generates $H_1$ in two steps, thus $H_1 \leq H'_1$, as required.
\end{proof}

Thus, we just proved that the following diagram makes sense, and commutes:

\begin{center}
\begin{tikzcd} G^{00} \arrow[d, "f"] \arrow[rd, "\iota_H \circ f"] \arrow[rr, "\iota_G"]&& G_1 \arrow[d, "\psi"] \\ H^{00} \arrow[r, "\iota_H"]
 & H_1 \arrow[r] & H_1 / H_0  \end{tikzcd}
\end{center}

It remains to prove that, in $\Ualg$, the morphism $\psi$ is surjective. Using commutativity of the outer square, and recalling that the $\mathcal{L}_0$-restriction of ${(\iota_H \circ f)}_*(p)$ is generic in $H_1$, we deduce that the image of $\psi$ contains an generic of $H_1 / H_0$ in $\Ualg$. By connectedness of the latter group in $\Ualg$, this implies that $\psi$ is surjective, as required.
\end{proof}

\bibliography{ref}

@article {BMW11,
    AUTHOR = {Blossier, Thomas and Martin-Pizarro, Amador and Wagner, Frank
              O.},
     TITLE = {G\'eom\'etries relatives},
   JOURNAL = {J. Eur. Math. Soc. (JEMS)},
  FJOURNAL = {Journal of the European Mathematical Society (JEMS)},
    VOLUME = {17},
      YEAR = {2015},
    NUMBER = {2},
     PAGES = {229--258},
      ISSN = {1435-9855,1435-9863},
   MRCLASS = {03C45 (03C60 20A15)},
  MRNUMBER = {3317743},
MRREVIEWER = {H.\ Dugald\ Macpherson},
       DOI = {10.4171/JEMS/502},
       URL = {https://doi.org/10.4171/JEMS/502},
}

@article {EHM19,
    AUTHOR = {Ealy, Clifton and Haskell, Deirdre and Ma\v{r}\'{\i}kov\'{a}, Jana},
     TITLE = {Residue field domination in real closed valued fields},
   JOURNAL = {Notre Dame J. Form. Log.},
  FJOURNAL = {Notre Dame Journal of Formal Logic},
    VOLUME = {60},
      YEAR = {2019},
    NUMBER = {3},
     PAGES = {333--351},
      ISSN = {0029-4527},
   MRCLASS = {03C64 (03C60 12J10 12J25)},
  MRNUMBER = {3985616},
MRREVIEWER = {Jafar S. Eivazloo},
       DOI = {10.1215/00294527-2019-0015},
       URL = {https://doi-org.libaccess.lib.mcmaster.ca/10.1215/00294527-2019-0015},
}

@article {HR19,
    AUTHOR = {Hrushovski, Ehud and Rideau-Kikuchi, Silvain},
     TITLE = {Valued fields, metastable groups},
   JOURNAL = {Selecta Math. (N.S.)},
  FJOURNAL = {Selecta Mathematica. New Series},
    VOLUME = {25},
      YEAR = {2019},
    NUMBER = {3},
     PAGES = {Paper No. 47, 58},
      ISSN = {1022-1824,1420-9020},
   MRCLASS = {03C45 (12J25)},
  MRNUMBER = {3984105},
MRREVIEWER = {Amador\ Martin-Pizarro},
       DOI = {10.1007/s00029-019-0491-x},
       URL = {https://doi.org/10.1007/s00029-019-0491-x},
}

@book {NIPbook,
    AUTHOR = {Simon, Pierre},
     TITLE = {A guide to {NIP} theories},
    SERIES = {Lecture Notes in Logic},
    VOLUME = {44},
 PUBLISHER = {Association for Symbolic Logic, Chicago, IL; Cambridge
              Scientific Publishers, Cambridge},
      YEAR = {2015},
     PAGES = {vii+156},
      ISBN = {978-1-107-05775-3},
   MRCLASS = {03-02 (03C45 03C60 03C64 68Q32)},
  MRNUMBER = {3560428},
MRREVIEWER = {Alf Onshuus},
       DOI = {10.1017/CBO9781107415133},
       URL = {https://doi-org.libaccess.lib.mcmaster.ca/10.1017/CBO9781107415133},
}

@book {HHM08,
    AUTHOR = {Haskell, Deirdre and Hrushovski, Ehud and Macpherson, Dugald},
     TITLE = {Stable domination and independence in algebraically closed
              valued fields},
    SERIES = {Lecture Notes in Logic},
    VOLUME = {30},
 PUBLISHER = {Association for Symbolic Logic, Chicago, IL; Cambridge
              University Press, Cambridge},
      YEAR = {2008},
     PAGES = {xii+182},
      ISBN = {978-0-521-88981-0},
   MRCLASS = {03-02 (03C45 03C60 12J10)},
  MRNUMBER = {2369946},
MRREVIEWER = {Thomas Warren Scanlon},
}

@article {CS18,
    AUTHOR = {Chernikov, Artem and Simon, Pierre},
     TITLE = {Definably amenable {NIP} groups},
   JOURNAL = {J. Amer. Math. Soc.},
  FJOURNAL = {Journal of the American Mathematical Society},
    VOLUME = {31},
      YEAR = {2018},
    NUMBER = {3},
     PAGES = {609--641},
      ISSN = {0894-0347},
   MRCLASS = {03C45 (03C60 03C64 22F10 28D15 37B05)},
  MRNUMBER = {3787403},
MRREVIEWER = {Rahim Nazim Moosa},
       DOI = {10.1090/jams/896},
       URL = {https://doi-org.libaccess.lib.mcmaster.ca/10.1090/jams/896},
}

@article {Hru12,
    AUTHOR = {Hrushovski, Ehud},
     TITLE = {Stable group theory and approximate subgroups},
   JOURNAL = {J. Amer. Math. Soc.},
  FJOURNAL = {Journal of the American Mathematical Society},
    VOLUME = {25},
      YEAR = {2012},
    NUMBER = {1},
     PAGES = {189--243},
      ISSN = {0894-0347,1088-6834},
   MRCLASS = {03C45 (11P70)},
  MRNUMBER = {2833482},
MRREVIEWER = {Katrin\ U.\ Tent},
       DOI = {10.1090/S0894-0347-2011-00708-X},
       URL = {https://doi.org/10.1090/S0894-0347-2011-00708-X},
}

@article {Del88,
    AUTHOR = {Delon, Fran\c{c}oise},
     TITLE = {Extensions s\'{e}par\'{e}es et imm\'{e}diates de corps valu\'{e}s},
   JOURNAL = {J. Symbolic Logic},
  FJOURNAL = {The Journal of Symbolic Logic},
    VOLUME = {53},
      YEAR = {1988},
    NUMBER = {2},
     PAGES = {421--428},
      ISSN = {0022-4812},
   MRCLASS = {12J10},
  MRNUMBER = {947849},
MRREVIEWER = {G. Cherlin},
       DOI = {10.2307/2274514},
       URL = {https://doi-org.libaccess.lib.mcmaster.ca/10.2307/2274514},
}

@article {KRV24,
    AUTHOR = {P.C Kovacsics and S. Rideau-Kikuchi and M. Vicaría},
     TITLE = {On orthogonal types to the value group and residual domination},
      YEAR = {2024},
       URL = https://arxiv.org/html/2410.22712v1}

@article {SV24,
    AUTHOR = {P. Simon and M. Vicaría},
     TITLE = {On Descent and germs},
      YEAR = {2024},
       URL = https://arxiv.org/abs/2407.19336}

@article {MOS20,
    AUTHOR = {Montenegro, Samaria and Onshuus, Alf and Simon, Pierre},
     TITLE = {Stabilizers, {${\rm NTP}_2$} groups with f-generics, and {PRC}
              fields},
   JOURNAL = {J. Inst. Math. Jussieu},
  FJOURNAL = {Journal of the Institute of Mathematics of Jussieu. JIMJ.
              Journal de l'Institut de Math\'{e}matiques de Jussieu},
    VOLUME = {19},
      YEAR = {2020},
    NUMBER = {3},
     PAGES = {821--853},
      ISSN = {1474-7480,1475-3030},
   MRCLASS = {03C45 (03C60 03C98)},
  MRNUMBER = {4094708},
MRREVIEWER = {Assaf\ Hasson},
       DOI = {10.1017/s147474801800021x},
       URL = {https://doi.org/10.1017/s147474801800021x},
}

@article {Vic22,
    AUTHOR = {Vicaria, Mariana},
     TITLE = {Residue field domination in henselian valued fields of
equicharacteristic zero, https://arxiv.org/pdf/2109.08243.pdf},
      YEAR = {2022},
	}

@article {EHS23,
    AUTHOR = {Ealy, Clifton and Haskell, Deirdre and Simon, Pierre},
     TITLE = {Residue field domination in some henselian valued fields},
   JOURNAL = {Model Theory},
  FJOURNAL = {Model Theory},
    VOLUME = {2},
      YEAR = {2023},
    NUMBER = {2},
     PAGES = {255--284},
      ISSN = {2832-9058,2832-904X},
   MRCLASS = {03C60 (12J10 12L12)},
  MRNUMBER = {4662632},
MRREVIEWER = {Piotr\ L.\ Kowalski},
       DOI = {10.2140/mt.2023.2.255},
       URL = {https://doi.org/10.2140/mt.2023.2.255},
}

@article {ACG22,
    AUTHOR = {Aschenbrenner, Matthias and Chernikov, Artem and Gehret, Allen
              and Ziegler, Martin},
     TITLE = {Distality in valued fields and related structures},
   JOURNAL = {Trans. Amer. Math. Soc.},
  FJOURNAL = {Transactions of the American Mathematical Society},
    VOLUME = {375},
      YEAR = {2022},
    NUMBER = {7},
     PAGES = {4641--4710},
      ISSN = {0002-9947,1088-6850},
   MRCLASS = {03C45 (03C60 12J25 12L12)},
  MRNUMBER = {4439488},
MRREVIEWER = {Alexandre\ Ivanov},
       DOI = {10.1090/tran/8661},
       URL = {https://doi.org/10.1090/tran/8661},
}

@article{pillay-tanovic,
    author = "A. Pillay and P. Tanovi\'{c}",
    title = "Generic stability, regularity and quasi-minimality",
    year = 2011,
    journal = "Models, Logics and Higher-dimensional Categories, CRM Proceedings and Lecture Notes",
    volume = "53",
    pages = "189-211",
    }

@article{Wang_group_config,
  author = "Paul Wang",
  title = "The group configuration theorem for generically stable types",
  DOI={10.1017/jsl.2024.45}, 
  journal={The Journal of Symbolic Logic},  
  year = 2024,
  pages={1–44},
  }

@incollection {vdD89,
    AUTHOR = {van den Dries, Lou},
     TITLE = {Dimension of definable sets, algebraic boundedness and
              {H}enselian fields},
      NOTE = {Stability in model theory, II (Trento, 1987)},
   JOURNAL = {Ann. Pure Appl. Logic},
  FJOURNAL = {Annals of Pure and Applied Logic},
    VOLUME = {45},
      YEAR = {1989},
    NUMBER = {2},
     PAGES = {189--209},
      ISSN = {0168-0072,1873-2461},
   MRCLASS = {03C45 (03C60)},
  MRNUMBER = {1044124},
MRREVIEWER = {Anand\ Pillay},
       DOI = {10.1016/0168-0072(89)90061-4},
       URL = {https://doi.org/10.1016/0168-0072(89)90061-4},
}

@article{HruPil-GpPFF,
	author = {Hrushovski, Ehud and Pillay, Anand},
	journal = {Israel Journal of Mathematics},
	number = {1--3},
	pages = {203--262},
	title = {Groups definable in local fields and pseudo-finite fields.},
	volume = {85},
	year = {1994}}

@article {Mar18,
    AUTHOR = {Marker, David},
     TITLE = {Model Theory of Valued Fields},
  URL ={http://homepages.math.uic.edu/~marker/valued_fields.pdf},
      YEAR = {2018},
	}

@book {La02,
    AUTHOR = {Lang, Serge},
     TITLE = {Algebra},
    SERIES = {Graduate Texts in Mathematics},
    VOLUME = {211},
   EDITION = {third},
 PUBLISHER = {Springer-Verlag, New York},
      YEAR = {2002},
     PAGES = {xvi+914},
      ISBN = {0-387-95385-X},
   MRCLASS = {00A05 (15-02)},
  MRNUMBER = {1878556},
       DOI = {10.1007/978-1-4613-0041-0},
       URL = {https://doi.org/10.1007/978-1-4613-0041-0},
}

@article {BYC14,
    AUTHOR = {Ben Yaacov, Ita\"i{} and Chernikov, Artem},
     TITLE = {An independence theorem for {${\rm NTP}_2$} theories},
   JOURNAL = {J. Symb. Log.},
  FJOURNAL = {The Journal of Symbolic Logic},
    VOLUME = {79},
      YEAR = {2014},
    NUMBER = {1},
     PAGES = {135--153},
      ISSN = {0022-4812,1943-5886},
   MRCLASS = {03C45},
  MRNUMBER = {3226015},
MRREVIEWER = {Alf\ Onshuus},
       DOI = {10.1017/jsl.2013.22},
       URL = {https://doi.org/10.1017/jsl.2013.22},
}
\bibliographystyle{alpha}

\end{document}